%% file: premio_arxiv.tex
\documentclass{article}
\usepackage[english]{babel}
\usepackage{amsthm,amsmath,amssymb}
\usepackage{float, comment}
\usepackage{geometry}
\usepackage{graphicx}
\usepackage{multirow}
\usepackage{algorithmic}
\usepackage{algorithm}
\usepackage{listings}
\usepackage{xcolor}
\usepackage{booktabs}
\usepackage{abstract}
\usepackage[colorlinks, linkcolor=red, anchorcolor=blue, citecolor=green]{hyperref}
\usepackage{appendix}
\usepackage{bm}
\usepackage{authblk}
\usepackage{enumitem}
\usepackage{cleveref}
\usepackage{minitoc}

\newcommand{\assign}{:=}

\newcommand{\tmstrong}[1]{\textbf{#1}}
\newcommand{\tmverbatim}[1]{\text{{\ttfamily{#1}}}}
\newcommand{\tmop}[1]{\ensuremath{\operatorname{#1}}}

\newcommand{\x}{\mathbf{x}}
\newcommand{\y}{\mathbf{y}}

\newcommand{\pmvb}{\texttt{PMVB}}

\newcommand{\GPO}{\texttt{PMVB}}

\include{macros}

\providecommand{\lemmaname}{Lemma}
\providecommand{\remarkname}{Remark}
\providecommand{\theoremname}{Theorem}

\theoremstyle{plain}
\newtheorem{lem}{\protect\lemmaname}[section]
\theoremstyle{remark}
\newtheorem{rem}{\protect\remarkname}[section]
\theoremstyle{plain}
\newtheorem{thm}{\protect\theoremname}[section]
\theoremstyle{plain}

\theoremstyle{plain}

\providecommand{\corollaryname}{Corollary}
\theoremstyle{plain}

\crefdefaultlabelformat{#2\textbf{#1}#3} 
\crefname{section}{\textbf{section}}{\textbf{sections}}
\Crefname{section}{\textbf{Section}}{\textbf{Sections}}
\crefname{thm}{\textbf{theorem}}{\textbf{theorems}}
\Crefname{thm}{\textbf{Theorem}}{\textbf{Theorems}}
\crefname{lem}{\textbf{lemma}}{\textbf{lemmas}}
\Crefname{lem}{\textbf{Lemma}}{\textbf{Lemmas}}
\crefname{prop}{\textbf{proposition}}{\textbf{propositions}}
\Crefname{prop}{\textbf{Proposition}}{\textbf{Propositions}}
\crefname{algorithm}{\textbf{algorithm}}{\textbf{algorithms}}
\Crefname{algorithm}{\textbf{Algorithm}}{\textbf{Algorithms}}
\crefname{coro}{\textbf{Corollary}}{\textbf{corollaries}}
\Crefname{coro}{\textbf{Corollary}}{\textbf{corollaries}}
\crefname{definition}{\textbf{Definition}}{\textbf{definitions}}
\Crefname{definition}{\textbf{Definition}}{\textbf{definitions}}
\crefname{table}{\textbf{Table}}{\textbf{tables}}
\Crefname{table}{\textbf{Table}}{\textbf{tables}}
\crefname{figure}{\textbf{Figure}}{\textbf{figures}}
\Crefname{figure}{\textbf{Figure}}{\textbf{figures}}

\author[1]{Yanguang Chen\thanks{2017212301@live.sufe.edu.cn}}
\author[2]{Wenzhi Gao\thanks{gwz@stanford.edu, equal contribution}}
\author[3]{Wanyu Zhang\thanks{zwanyu@stanford.edu, equal contribution}}
\author[4]{Dongdong Ge\thanks{dongdong@gmail.com}}
\author[4]{Huikang Liu\thanks{hkliu@gmail.com}}
\author[2,3]{Yinyu Ye\thanks{yyye@stanford.edu}}
\affil[1]{Shanghai University of Finance and Economics}
\affil[2]{ICME, Stanford University}
\affil[3]{Management Science and Engineering, Stanford University}
\affil[4]{Antai College of Economics and Management, Shanghai Jiao Tong University}

\title{Data-driven Mixed Integer Optimization through Probabilistic Multi-variable Branching}

\begin{document}

\maketitle

\input{abstract.tex}

\setlength{\parindent}{0pt}
\input{intro}
\input{mvbtheory}
\input{practical.tex}
\input{datafree.tex}
\input{exp_new.tex}
\input{conclusion}

\renewcommand \thepart{}
\renewcommand \partname{}

\bibliography{ref}
\bibliographystyle{plain}

\appendix

\clearpage

\input{appendix}

\end{document}

%% file: macros.tex
\global\long\def\vertiii#1{\left\vert \kern-0.25ex  \left\vert \kern-0.25ex  \left\vert #1\right\vert \kern-0.25ex  \right\vert \kern-0.25ex  \right\vert }%

\global\long\def\and{\mathrm{and}}%

\global\long\def\Ebb{\mathbb{E}}%

\global\long\def\Pbb{\mathbb{P}}%

%% file: abstract.tex
This paper introduces Probabilistic Multi-Variable Branching ({\GPO}), a simple yet highly flexible technique for accelerating mixed-integer optimization using data-driven machine learning models. At its core, {\GPO} employs a multi-variable cardinality branching procedure that partitions the feasible region with data-driven hyperplanes, requiring only two lines of code for implementation. Moreover, {\GPO} is model-agnostic and can be readily integrated with various machine learning approaches. Leveraging tools from statistical learning theory, we develop interpretable hyperparameter selection strategies to enhance its performance. Furthermore, we extend our approach to a data-free setting, where the root LP relaxation serves as a surrogate prediction model, and we provide theoretical analysis to justify this idea. We evaluate {\GPO} by incorporating it into state-of-the-art MIP solvers and conducting experiments on both classic benchmark datasets and real-world instances. The results demonstrate its effectiveness in significantly improving MIP-solving efficiency.

%% file: intro.tex
\section{Introduction} \label{sec:intro}
Mixed Integer Programming (MIP){\cite{sahinidis2019mixed, wolsey2020integer}} is a core modeling tool in applications such as revenue management {\cite{benati2007mixed}}, production planning {\cite{pochet2006production}}, portfolio optimization {\cite{wolsey1999integer}}, and energy management {\cite{wood2013power}}. Despite their broad applicability, MIPs are hard to solve, and state-of-the-art solvers {\cite{cplex2009v12, ge2022cardinal, gurobi2021gurobi,optimizer2021v8}} rely on heuristics to be useful in practice. Designing effective heuristics has been a long-standing topic in MIP solving.\\

Among various MIP applications, many exhibit a strong online nature: similar instances with varying data parameters are solved on a regular basis \cite{wood2013power, pochet2006production}. These applications open up the way for integrating machine learning (ML) techniques with MIP solving: ML models can extract solution patterns from historical data and use this information to accelerate repeated solving tasks. When applied to models with similar structures and varying data parameters, these methods often yield improved performance.\\

\texttt{ML+MIP} approaches typically start by designing or specifying a learning model to handle the MIP information. These approaches range from classic machine learning models to sophisticated deep network architectures. Examples include supervised methods such as $k$-nearest neighbors and support vector machine \cite{xavier2021learning}, tree-based models \cite{alvarez2017machine}, and the widely adopted graph neural networks (GNN) \cite{han2023gnn, liu2022learning}. Recently, there has been increasing interest in leveraging reinforcement learning techniques to accelerate MIP solving \cite{berthold2021learning, scavuzzo2022learning, song2020general, wu2021learning} (see \Cref{table_first} for a summary and further details).  These advanced learning approaches more efficiently capture the structure of MIP instances and undergird the success of these \texttt{ML+MIP} attempts.

\begin{table}[!ht]
\centering
\caption{Summary of \texttt{ML+MIP} approaches. GNN: graph neural networks; RL: reinforcement learning; DNN: Deep Neural Network; DQN: Deep Q-Network \label{table_first}}
    \begin{tabular}{ccc}
    \toprule
       Ref & Model & Application  \\ 
        \midrule
        \cite{nair2020solving} & GNN & Learn to dive and branch  \\ 
        \cite{xavier2021learning} & $k$-NN and SVM & Learn affine subspaces  \\ 
        \cite{nazari2018reinforcement} & RL &  Generate near-optimal solutions  \\ 
        \cite{khalil2022mip} & GNN & Learn variable biases  \\ 
        \cite{gasse2019exact} & GNN & Learn a variable selection policy  \\ 
        \cite{ding2020accelerating} & GNN & Learn to predict solutions  \\ 
        \cite{alvarez2017machine} & Extremely Randomized Trees & Learn to branch  \\ 
        \cite{han2023gnn} & GNN & Predict the variable probability  \\ 
        \cite{zarpellon2021parameterizing} & GNN & Learn to generalize branching  \\ 
    \cite{liu2022learning} & GNN & Learn neighborhood size for local branching  \\ 
        \cite{gupta2020hybrid} & GNN, DNN,  RL &  Develop a GNN-MIP hybrid model  \\ 
        \cite{sonnerat2021learning} & GNN, RL, DQN, DNN & Learn to select neighborhoods for LNS  \\ 
        \cite{qu2022improved} & RL, GNN & Reinforcement learning-based branching  \\ 
        \cite{lin2022learning} & Transformer & Tree-aware transformer-based branching  \\ 
        \bottomrule
    \end{tabular}
\end{table}

After a learning model is established, its output is used to speed up MIP solving, usually by modifying the behavior of the standard branch-and-bound procedure.
 In {\cite{alvarez2017machine,khalil2016learning}}, the authors use
machine learning models to imitate the behavior of an expensive branching rule
in MIP. {\cite{ding2020accelerating}} directly predict the values of the binary variables and proposes to add a cut that upper-bounds the distance between prediction and the MIP solution, but they do not further analyze this method. {\cite{nair2020solving}} also adopt a graph representation of
MIP to predict binary variables, and proposes to combine the prediction with MIP diving heuristic (neural
diving). In \cite{gasse2019exact}, the authors propose to learn the branching variable selection policy during the branch-and-bound procedure.
{\cite{lee2019learning}} adopts imitation learning to learn a good pruning policy for nonlinear MIPs.
Though most \tmverbatim{ML+MIP} studies are empirical, there exist works in
statistical learning theory that study the learnability of MIPs. In
{\cite{balcan2018learning}}, the authors derive the generalization bound for
learning hybrid branching rules that take convex combination of different
score functions. In {\cite{balcan2022improved}}, the bound is further
improved for the MIP branch and bound procedure. In the same line of work, \cite{balcan2022structural} analyze the learnability of mixed integer and Gomory cutting planes.\\

While existing \texttt{ML+MIP} methodologies have demonstrated promising results, they suffer from two key drawbacks. First, most available methods require modifications to the branch-and-bound procedure of standard MIP solvers (for example, \cite{nair2020solving,gasse2019exact}), making them challenging to deploy without access to the solver's source code or an advanced call-back interface. Second, these \texttt{ML}-based approaches often involve hyperparameters that are not easily interpretable, thereby demanding extensive manual tuning efforts.\\

In view of the challenges, in this paper we propose
{\GPO}, a Probabilistic
Multi-variable Branching approach for MIP solving. \tmverbatim{PMVB} relies on a probabilistic estimate of binary (integer) variables and applies a simple but efficient multi-variable cardinality branching procedure that constructs branching hyperplanes that take the form
\[\textstyle \sum_{i \in \mathcal{S}} y_i \leq \xi \quad \text{and} \quad \sum_{i \in
   \mathcal{S}} y_i \geq \xi + 1, \quad \{ y_i \} \text{ binary}.\]
Our method is simple yet effective, and in particular, our contributions include:\\

First, we develop a risk pooling technique to generate statistical inequalities (branching hyperplanes) on MIP binary variables using the probabilistic output of any ML prediction model. We show that the MIP optimal solution satisfies the inequalities with high probability under statistical assumptions \cite{vershynin2018high}. Adding these inequalities to the original problem could reduce the solving time at almost no additional implementation cost. The method is simple, effective, and compatible with most ML models/MIP solvers, and it can be applied as a standalone heuristic.\\

Based on the hyperplanes \cite{karamanov2011branching,fischetti2003local},  we propose \tmverbatim{PMVB}, a probabilistic multi-variable cardinality branching \cite{gamrath2015branching, yang2021multivariable} method to accelerate MIP solving. Specifically, we create two probabilistic hyperplanes for variables predicted to be 0 and 1. These two hyperplanes divide the problem into four subproblems. {\pmvb} first solves the subproblem with these two inequalities holding with high probability and gets a near-optimal objective value $c$. Then, we add the inequality that the objective value is no bigger than $c$ (for minimization) to the other three subproblems, so that they could be quickly pruned.\\

We extend our risk pooling idea to the ``data-free'' case, where no history data is available and ML prediction is inaccessible. In this case, we instead use the solution to the MIP root relaxation as the prediction result and apply {\pmvb}. Numerical results on various benchmark datasets show that {\pmvb} still yields a nontrivial improvement.

\subsection{Outline of the paper}
The paper is organized as follows. In  \Cref{sec:gpo}, we introduce {\GPO} and an underlying multi-variable cardinality branching procedure, giving a theoretical analysis of its behavior using statistical learning theory. In  \Cref{sec:prac}, we discuss the practical aspects when implementing {\GPO}. \Cref{sec:data-free} discusses the data-free variant of {\GPO} where we consider root LP solution as an ML model output. Finally, in \Cref{sec:exp} we conduct numerical experiments on both synthetic and real-life MIP instances to demonstrate the efficiency of our method.

%% file: mvbtheory.tex
\section{Data-driven probabilistic multi-variable branching}\label{sec:gpo}

This section introduces {\pmvb}, a data-driven probabilistic multi-variable cardinality branching procedure. In brief, cardinality branching partitions the feasible region by adding complementary hyperplanes involving multiple variables, while {\pmvb} adjusts the intercepts of these hyperplanes using the idea of risk pooling and concentration inequalities. We begin by formally defining parametric mixed-integer programs (MIPs) and reviewing the necessary preliminaries from statistical learning theory.

\subsection{Parametric MIP and solution map}
We consider a parametric MIP \cite{jenkins1982parametric} $\mathcal{P} (\xi)$ with $n$ binary variables, defined formally below. Throughout, we assume the problem with fixed binary variables is tractable and restrict our attention to problems only involving binary variables:
\begin{equation}\label{eq:prototype}
	\mathcal{P} (\xi) \assign \big[ \; \min_{\x \in \mathbb{R}^d, \, \y \in \{ 0, 1 \}^n} ~~c ( \x, \y ; \xi
  ) \quad \text{subject to} \quad h( \x, \y ; \xi ) \leq 0 \; \big]
\end{equation}
where the functions $c, h$ encapsulate the structural components of the problem, and $\xi \sim \Xi$ represents variable data parameters drawn from a distribution $\Xi$. We focus on efficiently solving a sequence of instances parameterized by $\xi$. Let $\y^{\star} (\xi) = (y_1^{\star}
(\xi), \ldots, y_n^{\star} (\xi))$ denote the corresponding optimal solution mapping. \\

In data-driven approaches for parametric MIPs, it is often assumed that a large collection of data-solution pairs $\{(\xi, \y^\star(\xi))\}$ is available, obtained either from historical solution records or through simulation. 
Given this data, one can train predictive models to capture and approximate $\y^\star(\xi)$, the underlying mapping from problem parameters $\xi$ to optimal solutions. This approach involves the following three main stages:

\paragraph{Workflow of standard data-driven approaches for MIP.}
\begin{enumerate}[leftmargin=30pt, label=\textbf{P\arabic*:}, ref=\textbf{P\arabic*}]
\item A set of solved instances $\mathcal{S}=\{(\xi, \y^\star{(\xi)})\}$, obtained from historical records or simulations. \label{P1}
\item Machine learning model $\hat{\y}$ trained on the dataset $\mathcal{S}$. \label{P2}
\item A mechanism that integrates the predictions $\hat{\y}(\xi)$ into the MIP solving process. \label{P3}
\end{enumerate}

Despite the wide adoption of this methodology, to our knowledge, considerable emphasis has been placed on the generation of predictive models \ref{P2}, while less attention has been paid to systematically integrating these predictions into the MIP solving process \ref{P3}. Motivated by this gap, the following subsections demonstrate how a cardinality branching procedure provides a simple yet effective solution to \ref{P3}, distinguishing our {\GPO} approach from prior \texttt{ML+MIP} frameworks. 


\subsection{Statistical learning theory for solving MIP}
We now explain how learning $\y^{\star}(\xi)$ fits naturally into statistical learning theory \cite{vapnik2013nature}. The following assumptions introduce randomness to the MIP prediction task.
\begin{enumerate}[start=1,label=\textbf{A\arabic*:},ref=\rm{\textbf{A\arabic*}}]
\item (i.i.d. distribution) The problem instances are i.i.d. generated according to parameters drawn from a data distribution $\xi \sim \Xi$. \label{A1}
\item (Binomial posterior estimate) Given $\xi \sim \Xi$, the MIP $\mathcal{P} (\xi)$ has one unique
solution, i.e., each binary decision variable is deterministically either 0 or 1
. \label{A2}
\end{enumerate}
Under assumption \ref{A2}, solving a MIP can be viewed as a binary classification task. For each $j$, a binary classifier $\hat{y}_j$ takes data parameter $\xi$ as input and produces a binary prediction. The optimal classifier in this context is the Bayes classifier defined by the conditional probability $\mathbb{P} \{ y^{\star}_j (\xi) = 1| \xi \}$. Consequently, a standard MIP solver can be viewed as one particular realization of such a Bayes classifier, which takes exponential time in the worst case.

However, classification by a MIP solver is computationally prohibitive. Instead, we select a suitable hypothesis class $\mathcal{Y}_j$ of candidate classifiers and apply empirical risk minimization (ERM) to learn an effective classifier $\hat{y}_j \in \mathcal{Y}_j$. The learning
procedure proceeds as follows: we collect offline data-solution pairs $\{ (\xi_i, \y^\star(\xi_i)) \}_{i=1}^m$, choose suitable function classes of classifiers
$\mathcal{Y}_j$, and compute the empirical risk minimizer $\hat{y}_j$ via:
\begin{equation*}
 \hat{y}_j \assign     \textstyle\underset{y_j \in \mathcal{Y}_j}{\arg \min}~  \big\{
    \frac{1}{m}  \sum_{i = 1}^m \mathbb{I} \{ y_j (\xi_i) \neq y^{\star}_j (\xi_i)
    \} \big\},\quad \text{and} \quad 
 e_j  \assign \displaystyle\min_{y_j \in \mathcal{Y}_j} \textstyle\big\{ \frac{1}{m} \sum_{i =
    1}^m \mathbb{I} \{ y_j (\xi_i) \neq y^{\star}_j (\xi_i) \} \big\},
\end{equation*}
where $\mathbb{I}\{\cdot\}$ is the 0-1 indicator function, $m$ is the number of solved instances, and  $e_j$ denotes ERM error. Given a
new instance $\mathcal{P} (\xi)$, the trained classifier $\hat{y}_j (\xi)$ is used to predict the optimal solution component $y_j^{\star} (\xi)$. In practice, a convex surrogate loss function \cite{reid2009surrogate} is employed instead of the indicator function to ensure the computational tractability of the ERM.

The performance guarantee for the classifier $\hat{y}_j$ can be characterized using standard results from statistical learning theory. Specifically, we have the following lemma:
\begin{lem}[\cite{vapnik2013nature}] \label{lem:learn-sol}
    Under {\ref{A1}} and {\ref{A2}}, letting $\hat{y}_j$ be the ERM
  classifier under 0-1 loss and $e_j$ be the ERM error, then with probability at least $1 - \delta$,
  \begin{equation}\label{ineq:vc}
      \mathbb{P}_{\xi \sim \Xi} \{ \hat{y}_j (\xi) \neq y^{\star}_j (\xi) \} \leq e_j
     + \sqrt{{\tfrac{1}{m}\big[\tmop{vc} (\mathcal{Y}_j) ( \log \tfrac{2 m}{\tmop{vc}
     (\mathcal{Y}_j)} + 1) + \log \tfrac{4}{\delta}}\big]},
  \end{equation} 
  where $\tmop{vc} (\mathcal{Y}_j)$ is the VC dimension of $\mathcal{Y}_j$.
\end{lem}
\Cref{lem:learn-sol} provides an upper bound on the prediction error of the classifier $\hat{y}_j$, showing that the upper bound improves as the sample size $m$ increases. To facilitate our analysis, we introduce an additional assumption ensuring that the trained classifier $\hat{y}_j$ is meaningful:
\begin{enumerate}[start=3,label=\textbf{A\arabic*:},ref=\rm{\textbf{A\arabic*}}]
\item (Learnability) The number of solved instances $m$ is large enough, such that for some properly chosen $\delta$ and for each $j \in [n]$, we have
\label{A3}
\begin{equation*}
\Delta_j^{\delta} \assign 1 - e_j - \sqrt{\tfrac{1}{m} \big[{\tmop{vc} (\mathcal{Y}_j)
   ( \log \tfrac{2 m}{\tmop{vc} (\mathcal{Y}_j)} + 1 ) + \log
   \tfrac{4}{\delta}}\big]} > 0.
\end{equation*}
\end{enumerate}
After establishing a probabilistic interpretation of $\hat{y}_j (\xi)$ and characterizing its performance via \Cref{lem:learn-sol}, we are ready to leverage these properties to introduce the {\pmvb} idea for efficiently solving parametric MIPs. 

\subsection{Risk-pooling and multi-variable cardinality branching}

Following the previous discussions, suppose we have obtained predictions $\hat{\y} = (\hat{y}_1, \ldots,
\hat{y}_n)$. Perhaps the most straightforward approach is to directly fix some binary variables according to these predictions. Specifically, for any instance $\mathcal{P}(\xi)$, we can set $y_j = \hat{y}_j (\xi)$ for each $j$ in a chosen subset $\mathcal{J} \subseteq [n]$. While this approach is straightforward, it quickly becomes unsafe -- a rough estimate shows that the probability of correctly fixing all variables decreases exponentially when more variables are involved:
\begin{equation*}
\textstyle \mathbb{P} \{ \hat{y}_j (\xi) = y_j^{\star} (\xi) , j \in \mathcal{J}
     \} \geq (1 - \delta)^{| \mathcal{J} |} \prod_{j \in \mathcal{J}}
     \Delta_j^{\delta},
\end{equation*}
where $\delta$ is defined in \ref{A3}. This observation highlights the inherent unreliability of the variable fixing strategies, where the risk of misclassification accumulates multiplicatively. 

To mitigate this issue, another idea is to apply \textit{risk pooling}, grouping these ``risky'' binary predictions together to reduce the overall probability of errors collectively. Specifically, we define the set $\mathcal{U} \assign \{
j : \hat{y}_j (\xi) = 1 \}$ and $\mathcal{L} \assign \{ j : \hat{y}_j (\xi) = 0 \}$
based on the predicted values. Using these sets, we construct two hyperplanes, referred to as probabilistic cardinality branching hyperplanes:
\begin{equation}\label{eqn:branching-hyperplanes}
\mathcal{C}_{\mathcal{U}}:  \textstyle \sum_{j \in \mathcal{U}} y_j \geq \zeta_1 \quad\text{and}\quad  \mathcal{C}_{\mathcal{L}}:\sum_{j \in \mathcal{L}} y_j \leq \zeta_2,
\end{equation}
where $\zeta_1$ and $\zeta_2$ are parameters that need to be determined. Unlike traditional multi-variable branching strategies in the MIP literature, the probabilistic interpretation of $\{ \hat{y}_j(\xi) \}$ provides a hint for selecting $\zeta_1, \zeta_2$ while characterizing the probabilistic behavior of the resulting four regions. To proceed with the analysis, we introduce the following assumption:
   \begin{enumerate}[start=4,label=\textbf{A\arabic*:},ref=\rm{\textbf{A\arabic*}}]
\item (Independent error rates) For any fixed $\xi \in \Xi$, the events $\{\hat{y}_j(\xi) = y^\star_j(\xi)\}$ for different indices $j \in [n]$ are independent. \label{A4}
\end{enumerate}
\begin{rem}
Assumption \ref{A4}, while quite strong, is still a standard assumption in the literature on ensemble methods in machine learning \cite{freund1997decision}. Here, we utilize it to highlight the effectiveness of the risk pooling concept. In \Cref{sec:prac}, we will demonstrate that this assumption can be removed without compromising the effectiveness of our approach.
\end{rem}

The probabilistic interpretation, combined with the independence assumption \ref{A4}, enables us to interpret the two hyperplanes in \eqref{eqn:branching-hyperplanes} as concentration inequalities:
\begin{thm} \label{thm:mvb}
Under the same assumptions as {\textbf{Lemma \ref{lem:learn-sol}}}, along with \ref{A3} and \ref{A4},
    consider a new MIP instance $\mathcal{P} (\xi)$ with $ \xi \sim \Xi$. Letting
  $\mathcal{U}= \{ j : \hat{y}_j (\xi) = 1 \}$ and $\mathcal{L}= \{ j : \hat{y}_j
  (\xi) = 0 \}$, for any $\gamma \geq 0$, we have
  \begin{equation*}
  \begin{aligned}
  	\textstyle \mathbb{P} \big\{ \sum_{j \in \mathcal{U}} y_j^{\star} (\xi) \geq (1-\delta)\sum_{j
     \in \mathcal{U}} \Delta_j^{\delta} - \gamma \big\} \geq{} & 1 - \exp \big(
     - \tfrac{2 \gamma^2}{| \mathcal{U} |} \big)\\
     \textstyle \mathbb{P} \big\{ \sum_{j \in \mathcal{L}} y_j^{\star} (\xi) \leq |
     \mathcal{L} | - (1-\delta)\sum_{j \in \mathcal{L}} \Delta_j^{\delta} + \gamma
     \big\} \geq{} & 1 - \exp \big( - \tfrac{2 \gamma^2}{| \mathcal{L} |}
     \big)
  \end{aligned}
   \end{equation*}
  where the probability is taken over both the training set and the new instance.
\end{thm}

\begin{proof}
For any $\xi \in \Xi$ and any $j \in \mathcal{U}$, we have
\begin{equation}\label{ineq:exp:ystar}
    \mathbb{E} [y_j^\star(\xi)] = \mathbb{P} \{ y^{\star}_j (\xi) = 1\} = \mathbb{P} \{ y^{\star}_j (\xi) = \hat{y}_j (\xi) \} \geq (1 - \delta)\Delta_j^{\delta},
\end{equation}
where the first equality holds because $y^{\star}_j (\xi)$ is a binary variable, the second equality holds because $\hat{y}_j (\xi) = 1$ for any $j \in \mathcal{U}$, and the last inequality holds because \Cref{lem:learn-sol} suggests that \eqref{ineq:vc} holds with probability at least $1 - \delta$. Then, based on \ref{A4},  we could apply the Hoeffding's inequality (see \textbf{Lemma \ref{lem:hoeffding}} in Appendix) to the sum $\sum_{j \in \mathcal{U}} y_j^{\star} (\xi)$ to get
\begin{equation}\label{ineq:sum:ystar}
    \textstyle \mathbb{P} \big\{ \sum_{j \in \mathcal{U}} y_j^{\star} (\xi) \geq \sum_{j \in \mathcal{U}} \mathbb{E} [y_j^\star(\xi)] - \gamma \big\} \geq{} 1 - \exp \big(
     - \tfrac{2 \gamma^2}{| \mathcal{U} |} \big).
\end{equation}
We complete the proof of the first part by combining \eqref{ineq:exp:ystar} and \eqref{ineq:sum:ystar} together. The proof of the second part is almost the same, so we omit it here.
\end{proof}

Similar to single-variable branching, these two hyperplanes $\mathcal{C}_{\mathcal{U}}$ and $\mathcal{C}_{\mathcal{L}}$---along with their complements $\bar{\mathcal{C}}_{\mathcal{U}}$ and $\bar{\mathcal{C}}_{\mathcal{L}}$---partition the feasible region into four disjoint subregions:  $\mathcal{P}_{\mathcal{C}_{\mathcal{U}}, \mathcal{C}_{\mathcal{L}}},
\mathcal{P}_{\mathcal{C}_{\mathcal{U}}, \bar{\mathcal{C}}_{\mathcal{L}}},
\mathcal{P}_{\bar{\mathcal{C}}_{\mathcal{U}}, \mathcal{C}_{\mathcal{L}}},
\mathcal{P}_{\bar{\mathcal{C}}_{\mathcal{U}},
   \bar{\mathcal{C}}_{\mathcal{L}}}$. This approach is recognized in the MIP literature as multi-variable cardinality branching or local branching \cite{fischetti2003local, gamrath2015branching}. 

\input{algo.tex}

By adjusting the intercepts of the hyperplanes based on concentration inequalties,  {\pmvb} ensures the optimal solution lies within $\mathcal{P}_{\mathcal{C}_{\mathcal{U}}, \mathcal{C}_{\mathcal{L}}}$ (the region defined by $\mathcal{C}_{\mathcal{U}}$ and $\mathcal{C}_{\mathcal{L}}$) with high probability. Although the bound provided by \textbf{Theorem \ref{thm:mvb}} may be weak for practical applications, it offers valuable theoretical insight and justification for our approach. The resulting algorithm, outlined in \textbf{Algorithm \ref{alg1}}, operates by solving the MIP within the four regions induced by the two hyperplanes. Since {\pmvb} is implemented by simply adding constraints, it can be integrated into any MIP solver with minimal modifications, requiring only a few lines of extra code.
\begin{rem} \label{rem:mvb}
One advantage of \textbf{Algorithm \ref{alg1}} is that, in some real-life applications, the computation can be restricted to the first subproblem under the following conditions: 1) there is high confidence that $\mathcal{P}_{\mathcal{C}_{\mathcal{U}}, \mathcal{C}_{\mathcal{L}}}$ contains the optimal solution, or
2) $\mathcal{P}_{\mathcal{C}_{\mathcal{U}}, \mathcal{C}_{\mathcal{L}}}$ already yields a satisfying solution.
In the worst case, all four subproblems are solved, ensuring no loss of optimality. More broadly, {\pmvb} can be employed either as a branching strategy to solve the MIP exactly or as a heuristic to accelerate the search for a high-quality primal solution.
\end{rem}

In the next section, we introduce practical aspects of {\pmvb}. For {\pmvb} to work in practice, we do not need the aforementioned assumptions, and there is more freedom to exploit the probabilistic interpretation of binary predictions and the concentration inequalities. 

%% file: algo.tex
\begin{algorithm}[h] 
\caption{Probabilistic multi-variable cardinality branching}
	\begin{algorithmic}[1]
	\STATE{\textbf{input}: MIP data parameter $(\xi)$, data-driven MIP classifiers $\{\hat{y}_j\}$}
    
    \STATE{\textbf{begin}}
    
    \quad{\tmstrong{compute}} $\hat{\y}(\xi)$
    
    \quad{\tmstrong{let}} $\mathcal{U}= \{ i : \hat{y}_i = 1 \},
    \mathcal{L}= \{ i : \hat{y}_i = 0 \}$
    
    \quad{\tmstrong{generate}}
 $\mathcal{C}_{\mathcal{U}}, \mathcal{C}_{\mathcal{L}},
\bar{\mathcal{C}}_{\mathcal{U}}, \bar{\mathcal{C}}_{\mathcal{L}}$
and four subproblems $\{
\mathcal{P}_{\mathcal{C}_{\mathcal{U}}, \mathcal{C}_{\mathcal{L}}},
\mathcal{P}_{\mathcal{C}_{\mathcal{U}}, \bar{\mathcal{C}}_{\mathcal{L}}},
\mathcal{P}_{\bar{\mathcal{C}}_{\mathcal{U}}, \mathcal{C}_{\mathcal{L}}},
\mathcal{P}_{\bar{\mathcal{C}}_{\mathcal{U}},
    \bar{\mathcal{C}}_{\mathcal{L}}} \}$
    
    \quad{\tmstrong{solve}} $\mathcal{P}_{\mathcal{C}_{\mathcal{U}},
    \mathcal{C}_{\mathcal{L}}}$
    and get $\y^{\star}_{\mathcal{C}_{\mathcal{U}},
    \mathcal{C}_{\mathcal{L}}}$ \\
    \quad{(optional) \tmstrong{solve}}  $
\mathcal{P}_{\mathcal{C}_{\mathcal{U}}, \bar{\mathcal{C}}_{\mathcal{L}}},
\mathcal{P}_{\bar{\mathcal{C}}_{\mathcal{U}}, \mathcal{C}_{\mathcal{L}}},
\mathcal{P}_{\bar{\mathcal{C}}_{\mathcal{U}},
\bar{\mathcal{C}}_{\mathcal{L}}}$ 
    and get $\y^{\star}_{\mathcal{C}_{\mathcal{U}},
\bar{\mathcal{C}}_{\mathcal{L}}},
\y^{\star}_{\bar{\mathcal{C}}_{\mathcal{U}}, \mathcal{C}_{\mathcal{L}}},
\y^{\star}_{\bar{\mathcal{C}}_{\mathcal{U}},
    \bar{\mathcal{C}}_{\mathcal{L}}}$

    \STATE{\textbf{end}}
    
    {\tmstrong{output}} the best solution among $\big\{
\y^{\star}_{\mathcal{C}_{\mathcal{U}}, \mathcal{C}_{\mathcal{L}}},
\y^{\star}_{\mathcal{C}_{\mathcal{U}}, \bar{\mathcal{C}}_{\mathcal{L}}},
\y^{\star}_{\bar{\mathcal{C}}_{\mathcal{U}}, \mathcal{C}_{\mathcal{L}}},
\y^{\star}_{\bar{\mathcal{C}}_{\mathcal{U}},
    \bar{\mathcal{C}}_{\mathcal{L}}} \big\}$
		
\end{algorithmic}\label{alg1}
\end{algorithm}

%% file: practical.tex
\section{Practical aspects of {\GPO}}\label{sec:prac}
Building on the intuitions from \Cref{sec:gpo}, this section explores practical considerations to further improve {\pmvb}. In practice, the output of the machine learning prediction model is a probability (or more precisely, a likelihood) $\tilde{y}_j(\xi) \in [0, 1]$. Given some threshold $\tau \in [0.5, 1]$ and defined the rounded solution:
\begin{equation}\label{eq:predict}
    \hat{y}_j(\xi) = \left\{ \begin{array}{ll}
       1  &  \text{if } \tilde{y}_j(\xi) \geq \tau, \\
       0  & \text{if } \tilde{y}_j(\xi) \leq 1 - \tau, \\
       \text{not rounded \; }  & \text{otherwise}.
    \end{array} \right.
\end{equation}
For any fixed data parameter $\xi$ and threshold $\tau$, we define the index sets of variables predicted at their lower and upper bounds as:
\begin{equation}\label{def:set:LU}
    \mathcal{L} (\tau; \xi) = \{ j : \tilde{y}_j (\xi) \leq 1 - \tau \} \quad  \text{and} \quad
\mathcal{U} (\tau; \xi) = \{ j : \tilde{y}_j (\xi) \geq \tau \}.
\end{equation} 
Using the definition, we define the  prediction accuracies, which are two random variables, as:
\begin{equation}\label{def:alpha:LU}
    \alpha_{\mathcal{L}}(\tau; \xi) = \tfrac{\sum_{j \in
   \mathcal{L} (\tau; \xi)} \mathbb{I} \{ y_j^{\star} (\xi) = 0 \}}{|
   \mathcal{L} (\tau; \xi)|} \quad  \text{and} \quad \alpha_{\mathcal{U}}(\tau; \xi) = \tfrac{\sum_{j \in
   \mathcal{U} (\tau; \xi)} \mathbb{I} \{ y_j^{\star} (\xi) = 1 \} }{|
   \mathcal{U} (\tau; \xi)|}.
\end{equation}
The quantities $\alpha_{\mathcal{L}}(\tau; \xi)$ and $\alpha_{\mathcal{U}}(\tau; \xi)$ quantify the prediction accuracy of the model $\tilde{y}$ with any threshold $\tau$ for any given instance $\xi$. Different choices of the hyperparameter $\tau$ can result in varying prediction accuracies, thereby influencing the performance of {\pmvb}. \textbf{Figure \ref{fig:predict}} illustrates the mean and variance of $\alpha_{\mathcal{L}}(\tau; \xi)$ and $ \alpha_{\mathcal{U}} (\tau; \xi)$ over $\xi \in \Xi$, as well as the average sizes of the index sets $\mathcal{L} (\tau; \xi)$ and $ 
\mathcal{U} (\tau; \xi)$, as   functions of the threshold parameter $\tau$. These statistics are computed using predictions $\{\tilde{y}_j(\xi)\}$ obtained from a GNN model, based on 100 \texttt{independent set} instances (see \textbf{Section \ref{sec:exp}} for more details).
\begin{figure*}[!h]
	\centering
    \includegraphics[scale=0.4]{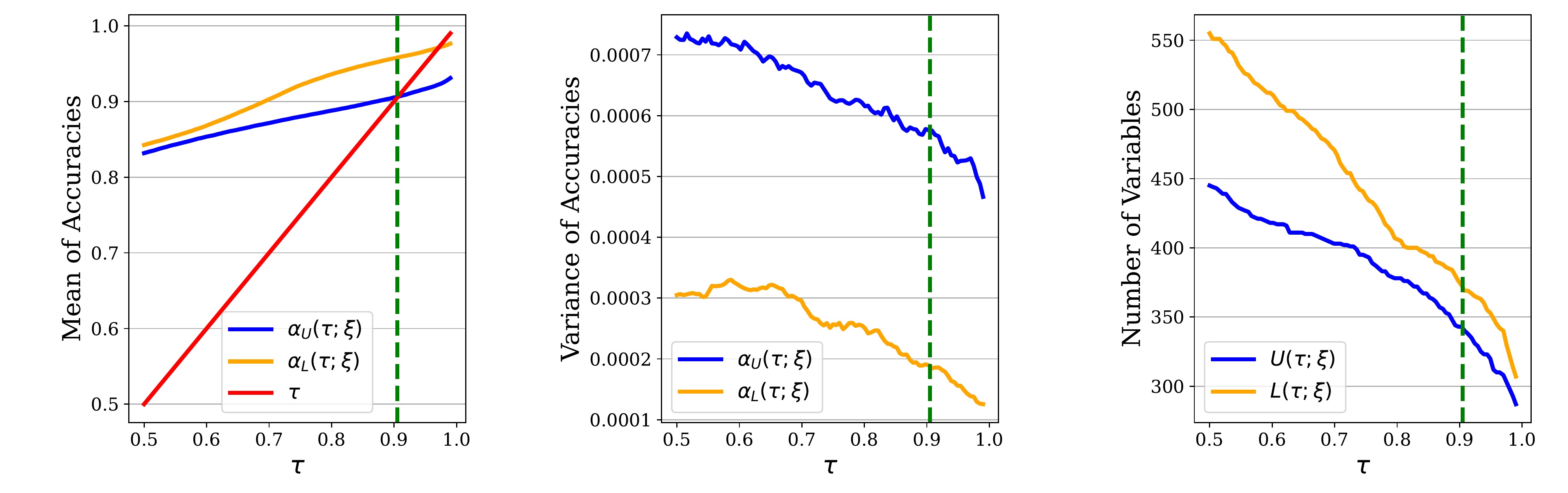}
	\caption{The mean and variance of  $\alpha_{\mathcal{L}}(\tau; \xi)$ and $ \alpha_{\mathcal{U}} (\tau; \xi)$, along with the average sizes of $\mathcal{L} (\tau; \xi)$ and $ 
\mathcal{U} (\tau; \xi)$, computed over 100 \texttt{independent set} instances for varying values of the threshold parameter $\tau$. \label{fig:predict}}
\end{figure*} 

As shown in \Cref{fig:predict}, there is a trade-off with respect to the choice of $\tau$. A larger $\tau$ tends to yield higher prediction accuracy and lower variance but includes fewer variables, potentially resulting in weaker hyperplanes. Conversely, a smaller $\tau$ includes more variables but leads to lower accuracy and higher variance. Thus, we recommend choosing the threshold $\tau$ according to:
\begin{equation}\label{ineq:tau}
    \tau^\star = \max_{\tau \in (0.5, 1]} \{ \tau : \mathbb{E} [\alpha_{\mathcal{L}}(\tau; \xi) ]  \geq \tau \text{ and } 
      \mathbb{E} [\alpha_{\mathcal{U}}(\tau; \xi)]  \geq \tau \}.
\end{equation}
The intuition behind the selection rule \eqref{ineq:tau} is to strike a balance between quantity (the number of variables included) and quality (the accuracy of predictions). Similar principles for balancing prediction quality and model complexity can be found in statistical model selection methods, such as Mallows's $C_p$ criterion \cite{gilmour1996interpretation}, $k$-fold cross validation \cite{bengio2004no}, and related techniques in reliability theory and extreme value theory. 

Let $\sigma^2$ denote an upper bound on the variance of both accuracies $\alpha_{\mathcal{L}}(\tau; \xi)$ and $\alpha_{\mathcal{U}} (\tau; \xi)$, that is,
\begin{equation}\label{ineq:bias}
      \mathbb{V} [\alpha_{\mathcal{L}}(\tau; \xi) ]  \leq \sigma^2 \quad \text{and} \quad 
      \mathbb{V} [\alpha_{\mathcal{U}}(\tau; \xi)]  \leq \sigma^2.
\end{equation}
By selecting the threshold $\tau$ according to the rule \eqref{ineq:tau} and applying Chebyshev's inequality (see \Cref{lem:chebyshev} in \Cref{appendix:lemma}) for any $\delta \in (0, 1)$, we obtain a probabilistic guarantee for prediction accuracy: 
\begin{equation}\label{ineq:chebyshev:alpha}
    \mathbb{P}[ |\alpha_{\mathcal{U}}(\tau; \xi) -  \mathbb{E} [\alpha_{\mathcal{U}}(\tau; \xi) ] |  \leq \tfrac{\sigma}{\sqrt{\delta}} ] \geq 1 - \delta.
\end{equation}
And a similar inequality holds for $\alpha_{\mathcal{L}}(\tau; \xi)$. Thus, combining \eqref{def:alpha:LU},  \eqref{ineq:tau} and \eqref{ineq:chebyshev:alpha} together, each of the following bound holds with probability at least $1 - \delta$:
$$
\textstyle \sum_{i \in \mathcal{U}(\tau^\star; \xi)}  y_i^{\star}(\xi) \geq
    \tau^\star | \mathcal{U}(\tau^\star; \xi) | - \tfrac{\sigma |\mathcal{U}(\tau^\star; \xi)|}{\sqrt{\delta}} , \quad \text{and} \quad \sum_{i \in \mathcal{L}(\tau^\star; \xi)}  y_i^{\star}(\xi) \leq
    (1 - \tau^\star) | \mathcal{L}(\tau^\star; \xi) | + \tfrac{\sigma |\mathcal{L}(\tau^\star; \xi)|}{\sqrt{\delta}}.
$$
Consider an example where the mean prediction accuracy is illustrated in \Cref{fig:predict}. Suppose we select the threshold $\tau = 0.9$, the confidence parameter $\delta = 0.05$, and estimate the variance bound $\sigma \approx 0.025$ according to \Cref{fig:predict}. If $\mathcal{U}(\tau; \xi) = \{1, 2, \dots, 100\}$, then the constraint becomes 
$$
\textstyle \sum_{i = 1}^{100} y_i^\star(\xi) \geq 90 - 100 \cdot \tfrac{0.025}{\sqrt{0.05}} \geq
     78.
$$
That is, with probability no less than 95\%, at least 78 out of the 100 variables take 1.
  
\subsection{Practical probabilistic multi-variable cardinality branching}

We now give a more practical
version of \textbf{Theorem \ref{thm:mvb}}. 
\begin{thm} \label{thm:pracmvb}
  Let $\mathcal{U}(\tau; \xi)$ and $\mathcal{L}(\tau; \xi)$ be defined as in \eqref{def:set:LU}, with $\tau$ chosen according to \eqref{ineq:tau}, and suppose $\sigma$ satisfies \eqref{ineq:bias}. Then, for any $\delta \in (0, 1)$, we have
  \begin{equation*}
      \begin{aligned}
 \Pbb\{\textstyle \sum_{i \in \mathcal{U}(\tau; \xi)}  y_i^{\star}(\xi) \geq 
    \tau | \mathcal{U(\tau; \xi)} | - \frac{\sigma |\mathcal{U}(\tau; \xi)|}{\sqrt{\delta}} \}\geq 1 - \delta, \\
\Pbb \{ \textstyle \sum_{i \in \mathcal{L}(\tau; \xi)}  y_i^{\star}(\xi)  \leq 
    (1 - \tau) | \mathcal{L(\tau; \xi)} | + \frac{\sigma |\mathcal{L}(\tau; \xi)|}{\sqrt{\delta}} \}\geq
     1- \delta.
\end{aligned}
  \end{equation*}
\end{thm}
Since \Cref{thm:pracmvb} follows directly from Chebyshev's inequality, its proof is omitted. It is worth noting that, in contrast to \Cref{thm:mvb}, \Cref{thm:pracmvb} does not rely on any additional assumptions, making it more broadly applicable in practice. However, this generality comes at a cost: without those assumptions, the probability bound in \Cref{thm:pracmvb} no longer decays exponentially and is therefore more conservative.
\begin{rem}
    In practice, according to the definition of $\mathcal{U}(\tau; \xi)$ and $\mathcal{L}(\tau; \xi)$,
    $\sum_{i \in \mathcal{U}(\tau; \xi)} \tilde{y}_i(\xi) \geq \tau | \mathcal{U}(\tau; \xi) |$ and 
    $ \sum_{i \in \mathcal{L}(\tau; \xi)} \tilde{y}_i(\xi) \leq (1 - \tau) | \mathcal{L}(\tau; \xi) |$, 
    one may use the following two inequalities to get further tighter constraints: 
    \begin{equation*}
        \begin{aligned}
\mathcal{C}'_{\mathcal{U}} :\quad 
    &   \textstyle \sum_{i \in \mathcal{U}(\tau; \xi)} y_i(\xi) \geq \sum_{i \in \mathcal{U}(\tau; \xi)} \tilde{y}_i(\xi) 
     - \frac{\sigma |\mathcal{U}(\tau; \xi)|}{\sqrt{\delta}}, \\
\mathcal{C}'_{\mathcal{L}} :\quad 
    &   \textstyle \sum_{i \in \mathcal{L}(\tau; \xi)} y_i(\xi) \leq \sum_{i \in \mathcal{L}(\tau; \xi)} \tilde{y}_i(\xi) 
     + \frac{\sigma |\mathcal{L}(\tau; \xi)|}{\sqrt{\delta}}. 
    \end{aligned}
    \end{equation*}
\end{rem}

\subsection{Efficient pruning using objective cuts} \label{sec:prune}

When {\GPO} is applied to solve MIP exactly as a branching rule, it is necessary to solve all four
subproblems to certify optimality. However, prior knowledge that the optimal solution more likely lies in $\mathcal{P}_{\mathcal{C}_{\mathcal{U}},
\mathcal{C}_{\mathcal{L}}}$ can be utilized to accelerate the solution procedure. Denote by
$c^{\star} (\{ \mathcal{C} \})$ the optimal value of MIP $\mathcal{P}(\xi)$ given in \eqref{eq:prototype} with extra
constraints $\{ \mathcal{C} \}$. Suppose we solve $\mathcal{P}_{\mathcal{C}_{\mathcal{U}},
\mathcal{C}_{\mathcal{L}}}$ and get the optimal objective
$c^{\star}_{\mathcal{U}, \mathcal{L}} \assign c^{\star} (\{
\mathcal{C}_{\mathcal{U}}, \mathcal{C}_{\mathcal{L}}\})$. Then we add
constraint $c ( \x, \y ; \xi
  ) \leq c^{\star}_{\mathcal{U}, \mathcal{L}}$ to
$\mathcal{P}_{\mathcal{C}_{\bar{\mathcal{U}}}, \mathcal{C}_{\mathcal{L}}}$,
$\mathcal{P}_{\mathcal{C}_{\mathcal{U}}, \mathcal{C}_{\bar{\mathcal{L}}}}$ and
$\mathcal{P}_{\mathcal{C}_{\bar{\mathcal{U}}},
\mathcal{C}_{\bar{\mathcal{L}}}}$ and solve
\begin{equation*}
\begin{aligned}
  c^{\star}_{\bar{\mathcal{U}}, \mathcal{L}}  
    &= c^{\star} \big( \{ \mathcal{C}_{\bar{\mathcal{U}}}, \mathcal{C}_{\mathcal{L}}, 
    \; c ( \x, \y ; \xi
  ) \leq c^{\star}_{\mathcal{U}, \mathcal{L}} \} \big), \\[5pt]
  c^{\star}_{\mathcal{U}, \bar{\mathcal{L}}}  
    &= c^{\star} \big( \{ \mathcal{C}_{\mathcal{U}}, \mathcal{C}_{\bar{\mathcal{L}}}, 
    \; c ( \x, \y ; \xi
  ) \leq c^{\star}_{\mathcal{U}, \mathcal{L}} \} \big), \\[5pt]
  c^{\star}_{\bar{\mathcal{U}}, \bar{\mathcal{L}}}  
    &= c^{\star} \big( \{ \mathcal{C}_{\bar{\mathcal{U}}}, \mathcal{C}_{\bar{\mathcal{L}}}, 
    \; c ( \x, \y ; \xi
  ) \leq c^{\star}_{\mathcal{U}, \mathcal{L}} \} \big).
\end{aligned}
\end{equation*}

Once $\mathcal{P}_{\mathcal{C}_{\mathcal{U}}, \mathcal{C}_{\mathcal{L}}}$ excludes all the optimal solutions to MIP,
$c^{\star}_{\bar{\mathcal{U}}, \mathcal{L}} = c^{\star}_{\mathcal{U},
\bar{\mathcal{L}}} = c^{\star}_{\bar{\mathcal{U}}, \bar{\mathcal{L}}} = +
\infty$ and other branches can be pruned efficiently, often at the cost of solving three relaxations problems.

\subsection{Compatibility with existing learning models}

As discussed in Remark \ref{rem:mvb}, {\GPO} is highly flexible in implementation. Once a predicted solution $\{\tilde{y}_j\}$ is available, {\GPO} can be applied with minimal additional effort. Its generality allows it to incorporate a wide range of learning models—such as those listed in \Cref{table_first}. In particular, in \Cref{sec:data-free}, we demonstrate the use of a data-free classifier based on the LP relaxation.

%% file: datafree.tex
\section{Data-free {\GPO} using LP relaxation}\label{sec:data-free}

Despite being a data-driven method, the idea behind {\GPO} is applicable so
long as a likelihood estimate of the binaries is available. In the MIP context, a
natural surrogate is the optimal solution to the root LP relaxation. Denote
$\mathbf{y}^{\star}_{\tmop{Root}}$ to be the MIP root relaxation solution
after root cut rounds, the idea is to construct hyperplanes based on
$\mathbf{y}^{\star}_{\tmop{Root}}$ instead. In other words, we view ``LP
solving'' itself as a prediction algorithm. We justify the idea by analyzing the theoretical properties of data-free {\pmvb} applied to a toy knapsack problem. 

\subsection{Analysis of data-free {\GPO}}

 We consider the following knapsack problem:
\begin{equation}\label{prob:knapsack}
    \begin{aligned}
        \max_{\y \in \{0, 1\}^n} \quad & c_1 y_1 + \cdots + c_n y_n \\
        \text{subject to} \quad & a_1 y_1 + \cdots + a_n y_n \leq b,
    \end{aligned}
\end{equation}
where $a_i, c_i \geq 0$ and $b >0$ are some given constants. We make the following assumptions:
\begin{enumerate}[start=5,label=\textbf{A\arabic*:},ref=\rm{\textbf{A\arabic*}}]
\item (Uniform cost) Parameters $\{a_i\}$ are i.i.d. uniformly distributed in $(0,1)$, i.e., $a_i \sim \mathcal{U}[0, 1]$. \label{A7}
\item (Uniform ratio) The weight $c_i = f_i \cdot a_i$ where $\{f_i\}$ are i.i.d. uniformly distributed in $[0,1]$, i.e., $f_i \sim \mathcal{U}[0, 1]$. \label{A8}
\end{enumerate}
Similar assumptions have also been made in \cite{lueker1982average} to analyze the average difference of objective value between the 0/1 knapsack problem and its linear relaxation. There are many other works on both the average and worst-case performance of LP relaxation for 0/1 knapsack problems, such as \cite{kohli2004average, morales2020expected}.
\begin{thm} \label{thm:data-free}
Let $\y^\star$ and $\tilde{\y}$ be the optimal solution to the knapsack problem \eqref{prob:knapsack} and its the LP relaxation, respectively. 
Under \ref{A7} and \ref{A8} and assuming $b = \gamma n$ for some constant $\gamma \in (0, \frac{1}{2})$, let
  $\mathcal{U} \assign \{ i : \tilde{y}_i = 1 \}$ and $\mathcal{L}= \{ i :
  \tilde{y}_i = 0 \}$. Then with probability at least $1 -\lceil 2\sqrt{n} \rceil \exp(\frac{-\sqrt{n}}{8})$, each of the inequalities
  below holds.
    \begin{equation*}
\begin{aligned}
 \mathcal{C}_{\mathcal{U}} :&     \textstyle \sum_{i \in \mathcal{U}}  y_i^{\star} \geq
    \sum_{i \in \mathcal{U}} \tilde{y}_i - 4\sqrt{2}n^{\frac{3}{4}} =  | \mathcal{U} | - 4\sqrt{2}n^{\frac{3}{4}}\\
    \mathcal{C}_{\mathcal{L}}: &    \textstyle \sum_{i \in \mathcal{L}}  y_i^{\star} \leq
    \sum_{i \in \mathcal{L}} \tilde{y}_i + 4\sqrt{2}n^{\frac{3}{4}} = 4\sqrt{2}n^{\frac{3}{4}}   
\end{aligned}
    \end{equation*}
\end{thm} 
\begin{proof}
The LP relaxation of \eqref{prob:knapsack} is given by
\begin{equation}\label{prob:knapsack-relaxation}
    \begin{aligned}
        \max_{\y \in \mathbb{R}^n} \quad & c_1 y_1 + \cdots + c_n y_n \\
        \text{subject to} \quad & a_1 y_1 + \cdots + a_n y_n \leq b, \\
        & 0 \leq y_i \leq 1, \, \forall i \in [n].
    \end{aligned}
\end{equation}
 We could write down its Lagrangian function as 
    $$
    L(\y, \lambda, \mathbf{\eta}, \mathbf{\xi}) = -(c_1 y_1 + \cdots + c_n y_n) + \lambda(a_1 y_1 + \cdots + a_n y_n - b) + \textstyle \sum_{i=1}^n \eta_i(y_i - 1) -\textstyle \sum_{i=1}^n \xi_i y_i
    $$
    with $\lambda \geq 0$, $\mathbf{\eta} \geq 0$ and $\mathbf{\xi} \geq 0$. Denote by $\tilde{\mathbf{y}}$ and $(\lambda^\star, \mathbf{\eta}^\star, \mathbf{\xi}^\star)$ the primal and dual optimal solutions to \eqref{prob:knapsack-relaxation}, which should satisfy the following KKT optimality conditions:
    \begin{equation*}
    \begin{aligned}
        a_i(f_i -  \lambda^\star) & = \eta^\star_i - \xi^\star_i, \, \forall i \\
        \eta^\star_i (y_i - 1) & = 0,  \, \forall i \\
        \xi^\star_i  y_i & = 0,  \, \forall i.
        \end{aligned}
    \end{equation*}
    Then, it is easy to see that the optimal solution $\tilde{\y}$ is given by
\begin{equation}\label{eq:opt:y}
    \tilde{y}_i \in 
    \begin{cases}
        \{1\} & \text{if } f_i > \lambda^\star, \\
        \{0\} & \text{if } f_i < \lambda^\star, \\
        [0, 1] & \text{otherwise}.
    \end{cases}
\end{equation}
    Without loss of generality, we assume $f_1 > f_2 > \cdots > f_n$. Recall the definition $\mathcal{U} \assign \{ i : \tilde{y}_i = 1 \}$ and $\mathcal{L}= \{ i :
  \tilde{y}_i = 0 \}$. Then, \eqref{eq:opt:y} implies that there exists some constant $k \in [n]$ such that
    \begin{equation*}
        \mathcal{U} = \{1, 2, \dots, k\} \quad \text{and} \quad \mathcal{L} = \{k + 2, \dots, n\}.
    \end{equation*}
    Here, we assume $\tilde{y}_{k+1}\in (0, 1)$. Otherwise, $\tilde{\y}$ is already the optimal solution to the original knapsack problem \eqref{prob:knapsack}, and the result trivially holds. Let $\mathcal{U}^\star \assign \{ i : y^\star_i = 1 \}$ and $\mathcal{L}^\star= \{ i : y^\star_i = 0 \}$. According to the optimality condition of $\y^\star$, we have
    \begin{equation*}
    \begin{aligned}
        c_1 y_1^\star + \cdots + c_n y_n^\star &= \textstyle \sum_{i \in \mathcal{U}^\star} c_i \geq \textstyle \sum_{i \in \mathcal{U}} c_i\\
         a_1 y_1^\star + \cdots + a_n y_n^\star &= \textstyle \sum_{i \in \mathcal{U}^\star} a_i \leq b \leq \textstyle \sum_{i \in \mathcal{U}} a_i + 1,
         \end{aligned}
    \end{equation*}
    where the last inequality holds since $b = a_1 \tilde{y}_1 + \cdots + a_n \tilde{y}_n = \textstyle \sum_{i \in \mathcal{U}} a_i + a_{k+1} \tilde{y}_{k+1} \leq \textstyle \sum_{i \in \mathcal{U}} a_i + 1$. In other words,
    \begin{equation}\label{ineq:ciai}
        \textstyle \sum_{i \in \mathcal{U}^\star \setminus \mathcal{U}} c_i \geq \textstyle \sum_{i \in \mathcal{U} \setminus \mathcal{U}^\star} c_i \quad \text{and} \quad \textstyle \sum_{i \in \mathcal{U}^\star \setminus \mathcal{U}} a_i \leq 1 + \textstyle \sum_{i \in \mathcal{U} \setminus \mathcal{U}^\star} a_i 
    \end{equation}
    Since $ \mathcal{U} = \{1, 2, \dots, k\}$, we know that $\mathcal{U}^\star \setminus \mathcal{U} \subseteq \{k+1, \cdots, n\}$. A simple computation shows  
    \begin{equation}\label{ineq:ci}
    \begin{aligned}
        \textstyle \sum_{i \in \mathcal{U}^\star \setminus \mathcal{U}} c_i = \textstyle \sum_{i \in \mathcal{U}^\star \setminus \mathcal{U}} f_i \cdot a_i & \leq (f_{k+1} - \delta) \textstyle \sum_{i \in \mathcal{U}^\star \setminus \mathcal{U}} a_i + \delta \cdot |\{i: f_i \in [f_{k+1} - \delta, f_{k+1}] \}|.
    \end{aligned}
    \end{equation}
    Note that for any given $\delta$, we know that $[f_{k+1} - \delta, f_{k+1}] \subset [(j-1)\delta, (j+1)\delta]$ for some $1 \leq j \leq 1/\delta$. Then, by applying  \Cref{lem:concentrate} with $2\delta$, we know that 
    \begin{equation}\label{ineq:set:gap}
         |\{i: f_i \in [f_{k+1} - \delta, f_{k+1}] \}| \leq 4n\delta
    \end{equation}
    holds with probability at least $1 -\lceil \frac{1}{\delta} \rceil \exp(\frac{-n\delta}{2})$. Combining \eqref{ineq:ciai}, \eqref{ineq:ci} and \eqref{ineq:set:gap} together, we get
    \begin{equation}
        \textstyle \sum_{i \in \mathcal{U}^\star \setminus \mathcal{U}} c_i \leq (f_{k+1} - \delta) \textstyle \sum_{i \in \mathcal{U}^\star \setminus \mathcal{U}} a_i + 4\delta^2 n \leq (f_{k+1} - \delta) (1 + \textstyle \sum_{i \in \mathcal{U} \setminus \mathcal{U}^\star} a_i) + 4\delta^2 n.
    \end{equation}
    Next, the fact that $\textstyle \sum_{i \in \mathcal{U}^\star \setminus \mathcal{U}} c_i \geq \textstyle \sum_{i \in \mathcal{U} \setminus \mathcal{U}^\star} c_i \geq f_k \cdot\textstyle \sum_{i \in \mathcal{U} \setminus \mathcal{U}^\star} a_i$ implies
    $$
    f_k \cdot\textstyle \sum_{i \in \mathcal{U} \setminus \mathcal{U}^\star} a_i \leq (f_{k+1} - \delta) (1 + \textstyle \sum_{i \in \mathcal{U} \setminus \mathcal{U}^\star} a_i ) + 4\delta^2 n \leq 1 + (f_{k} - \delta) \textstyle \sum_{i \in \mathcal{U} \setminus \mathcal{U}^\star} a_i + 4\delta^2 n.
    $$
    A simple computation shows that
    $$
    \textstyle \sum_{i \in \mathcal{U} \setminus \mathcal{U}^\star} a_i \leq \frac{1}{\delta} + 4\delta n = 4\sqrt{n}
    $$
    where the last inequality holds by choosing $\delta = \frac{1}{2\sqrt{n}}$. Lastly, by applying \Cref{lem:concentrate} with $\delta = \frac{1}{2\sqrt{n}}$ and $S = \mathcal{U} \setminus \mathcal{U}^\star$, we get $|\mathcal{U} \setminus \mathcal{U}^\star| \leq 4\sqrt{2}n^{\frac{3}{4}}$. 
    Thus, we complete the proof by
    $$
    \textstyle \sum_{i \in \mathcal{U}}  y_i^{\star} \geq \textstyle \sum_{i \in \mathcal{U} \cap \mathcal{U}^\star}  y_i^{\star}  = \textstyle \sum_{i \in \mathcal{U} \cap \mathcal{U}^\star}  \tilde{y}_i = \textstyle \sum_{i \in \mathcal{U}} \tilde{y}_i - |\mathcal{U} \setminus \mathcal{U}^\star| \geq
    \textstyle \sum_{i \in \mathcal{U}} \tilde{y}_i - 4\sqrt{2}n^{\frac{3}{4}}
    $$
    where the first equality holds since $y_i^{\star}  =  \tilde{y}_i = 1$ for all $i \in \mathcal{U} \cap \mathcal{U}^\star$. Another inequality w.r.t. $\mathcal{L}$ trivally holds since $\mathcal{L}^\star$ and $\mathcal{L}$ are the complementary set of $\mathcal{U}^\star$ and $\mathcal{U}$ (except for one element), respectively.
\end{proof}

\Cref{thm:data-free} shows that, for the standard 0/1-knapsack problem and under the uniform distribution setting, the similar generated hyperplane holds with high probability. Note that it is easy to see that $|\mathcal{U}| \geq b = \gamma n$, so that $4\sqrt{2}n^{\frac{3}{4}} = o(|\mathcal{U}|)$ is a small term with respect to the size of $\mathcal{U}$. 

\subsection{Using interior point solution in {\pmvb}} 

In practice, it is observed that the branching hyperplanes generated by the
interior point solution are often more effective than using the simplex
method. Simplex method always yields a vertex solution where most binary
variables take exactly 0/1, while interior point method gives fractional
values more accurately reflecting the likelihood of the variables taking 0/1.
Similar advantages of the interior point method have also been observed in the
literature {\cite{berthold2018four}}.

%% file: exp_new.tex
\section{Numerical experiment}\label{sec:exp}

In this section, we present experiments to evaluate the practical performance of different variants of \tmverbatim{\pmvb}, available on Github\footnote{The code is available on \href{https://github.com/zwyhahaha/GNN_MVB}{https://github.com/zwyhahaha/GNN\_MVB}}. For the data-driven experiments, we use two commercial solvers: \tmverbatim{Gurobi} \cite{gurobi2021gurobi} and \tmverbatim{COPT} \cite{ge2022cardinal}. Due to license restrictions, only \tmverbatim{COPT} is used for the MIPLIB experiments. All experiments are conducted on a machine with an Intel(R) Xeon(R) CPU E5-2680 @ 2.70GHz and 64GB of memory. 

\subsection{Experiment setup}
\paragraph{Learning model.} We adopt two types of learning models: simple logistic regression and graph neural network. Logistic regression is applied to tasks where only part of the coefficient in MIP changes, and it has a low VC dimension. We use a separate regression model for each binary variable and train $\hat{y}_i$ over the whole training set. The probabilistic output is used for
generating branching hyperplanes.  GNN is used for more complex scenarios where the problem size and data coefficient change. We reuse the GNN models from
{\cite{canturk2024scalable}} trained on 1000 instances and validated on 200
instances from each class.

\paragraph{MIP data.} For logistic regression, we use two combinatorial problems \textbf{1)}. Multi-knapsack problem
(\tmverbatim{MKP}). \textbf{2)}. Set covering problem (\tmverbatim{SCP}) and
a real-life problem \textbf{3)}. Security-constrained unit commitment
(\tmverbatim{SCUC}) from \tmverbatim{IEEE 118} as our testing models:
\begin{itemize}[leftmargin=*]
	\item For \tmverbatim{MKP} models $\max_{\mathbf{Ay} \leq \mathbf{b}, \mathbf{y}
\in \{ 0, 1 \}^n} \langle \mathbf{c}, \mathbf{y} \rangle$, $\mathbf{A} \in
\mathbb{R}^{m \times n}$,  we generate synthetic datasets according to the problem setup
in {\cite{chu1998genetic}}. After choosing problem sizes $(m, n) \in \{ (10,
250), (10, 500), (30, 250), (30, 500) \}$, each element of $a_{i j}$ is
uniformly drawn from $\{ 1, \ldots, 1000 \}$ and $c_i = \frac{1}{m} \sum_i
a_{i j} + \delta_i$ where $\delta_i$ is sampled from $\{ 1, \ldots, 500 \}$.
Then for each setup we fix $\mathbf{A}, \mathbf{c}$ and generate $500$ i.i.d.
instances with $b_i \sim \mathcal{U} [ 0.8 \cdot \frac{1}{4 n}\sum_j a_{i
j}, 1.2 \cdot \frac{1}{4 n} \sum_j a_{i j}]$ for training. We generate 20
new instances for testing. In \tmverbatim{MKP} $\xi = \mathbf{b}$.
\item For \tmverbatim{SCP} models $\min_{\mathbf{Ay} \geq \textbf{1}, \mathbf{y} \in
\{ 0, 1 \}^n} \langle \mathbf{c}, \mathbf{y} \rangle, \mathbf{A} \in
\mathbb{R}^{m \times n}$, we use dataset from {\cite{umetani2017exploiting}}.
After choosing problem sizes $(m, n) \in \{ (1000, 10000), (1000, 20000),
(2000, 20000) \}$ and generating $\mathbf{A}$ with sparsity 0.01, $\bar{c}_i
\in \{ 1, \ldots, 100 \}$, for each setup we fix $\mathbf{A}$ and generate
$500$ i.i.d. instances with $c_i \sim \mathcal{U} [0.8 \cdot \bar{c}_i, 1.2
\cdot \bar{c}_i]$ for training. We generate 20 new instances for testing. In
\tmverbatim{SCP} $\xi = \mathbf{c}$.

\item For \tmverbatim{SCUC} models $\min_{(\mathbf{x}, \mathbf{y}), h (\mathbf{x},
\mathbf{y}) \leq \textbf{0}, \mathbf{Dx} = \mathbf{d}} c (\mathbf{x},
\mathbf{y})$, we use \tmverbatim{IEEE 118} dataset
{\cite{birchfield2016grid,xu2017application}} of time horizon 96, fix all the
information of generators and simulate load demand vector $\mathbf{d} \in
\mathcal{U} [l, u]^{96}, (l, u) \in \{(3000, 5000), (4000, 6000)\}$. For each setup we generate 200 i.i.d. instances for training and
5 new instances for testing. In \tmverbatim{SCUC} $\xi = \mathbf{d}$.
\end{itemize}

For GNN, we use types of combinatorial problems \textbf{1)}
set covering (\tmverbatim{SCP}) \textbf{2)} combinatorial auction (\tmverbatim{CA}) \textbf{3)}
independent set (\tmverbatim{IS}). These three datasets are taken from {\cite{gasse2019exact}}. Aside from the datasets above, we also use MIPLIB 2017 collection
{\cite{gleixner2021miplib}} to test data-free {\pmvb}.   We take 132 instances from the MIPLIB benchmark set {\cite{gleixner2021miplib}}.

\paragraph{Training data.}
In logistic regression tasks, we build training data by ourselves: each
\tmverbatim{MKP} is given 300 seconds and solved to a $10^{- 4}$ relative gap;
Each \tmverbatim{SCP} of size $(m, n) = (1000, 10000)$ is given 300 seconds and
solved to $10^{- 3}$ gap; Each \tmverbatim{SCP} of size $(m, n) \in \{ (1000,
20000), (2000, 20000) \}$ is given 600 seconds and solved to a $10^{- 1}$ gap;
Each \tmverbatim{SCUC} is given 3600 seconds and solved to a $10^{- 4}$ gap. For GNN, 1000 instances are
generated according to the setup in
{\cite{canturk2024scalable,gasse2019exact,gleixner2021miplib}}. The labels
used for training are generated using \tmverbatim{CPLEX}.

\subsection{{\pmvb} as a primal heuristic}

In this section, we benchmark the performance of {\GPO} as a primal heuristic. We choose two current leading MIP solvers \tmverbatim{Gurobi} and \tmverbatim{COPT} from MIP benchmark {\cite{mittelmann2020benchmarking}} and use them to evaluate performance improvement using {\GPO} to improve primal upper bound. 

Given a
new MIP instance from our test set, we construct the two hyperplanes according
to \textbf{Theorem \ref{thm:pracmvb}} and use hyperparameters $\delta = 0.8, \tau=0.9$. After
generating the hyperplanes, we allow 120 seconds for the solver to solve
$\mathcal{P}_{\mathcal{C}_{\mathcal{L}}, \mathcal{C}_{\mathcal{U}}}$ and
obtain the time $T_{\text{\tmverbatim{\pmvb}}}$ it spends achieving the best
objective value $F_{\text{\tmverbatim{\pmvb}}}$. Then we run the solver again, with a time limit of 3600 seconds, on the MIP instance without adding hyperplanes and record the time
$T_{\text{\tmverbatim{Original}}}$ for the solver to reach
$F_{\text{\tmverbatim{\pmvb}}}$. We take the shifted geometric mean (SGM) \cite{mittelmann2020benchmarking} over testing instances and compute
\begin{equation*}
    \begin{aligned}
\textstyle \bar{T}_{\text{\tmverbatim{\pmvb}}} &:=  \exp({ \textstyle \frac{1}{n}\sum_{i=1}^n} \log \max\{1, T^i_{\text{\tmverbatim{\pmvb}}} + 10\}) - 10,\\
\textstyle \bar{T}_{\text{\tmverbatim{Original}}} &:=  \exp({ \textstyle \frac{1}{n}\sum_{i=1}^n} \log \max\{1, T^i_{\text{\tmverbatim{Original}}} + 10\}) - 10.	
\end{aligned}
\end{equation*}

We report the average speedup \(\Sigma \assign \frac{\bar{T}_{\text{\tmverbatim{Original}}}}{\bar{T}_{\text{\tmverbatim{\pmvb}}}}\) as the primary performance evaluation metric. To assess how our method interacts with the built-in heuristics of MIP solvers, we vary the solvers' heuristic-related parameters across different levels. Specifically, we use \texttt{L}, \texttt{M}, and \texttt{H} to denote low, medium, and high heuristic settings, respectively. For \tmverbatim{Gurobi}, these correspond to \texttt{Heuristics} values \(\{0, 0.05, 1.0\}\); for \tmverbatim{COPT}, they correspond to \texttt{HeurLevel} values \(\{0, -1, 3\}\). Here, \texttt{L} disables heuristics, \texttt{M} uses the default setting, and \texttt{H} enables heuristics aggressively.\\
 
\begin{figure*}[h!]
\centering
\includegraphics[scale=0.25]{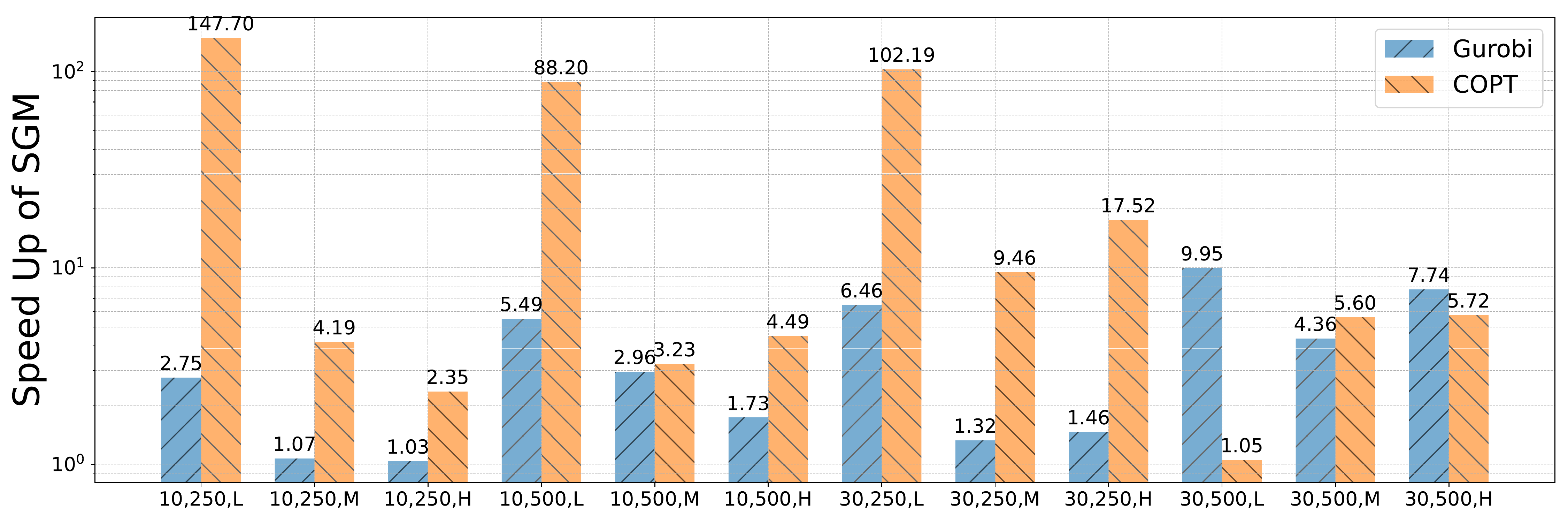}	
\includegraphics[scale=0.25]{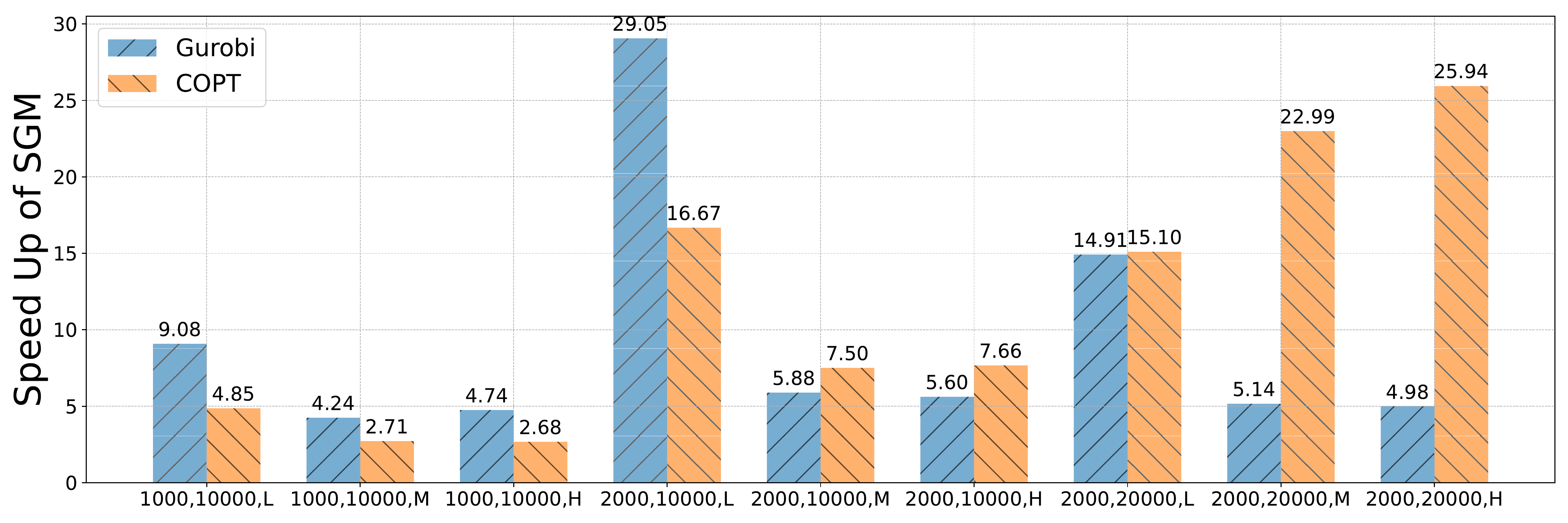}
\includegraphics[scale=0.25]{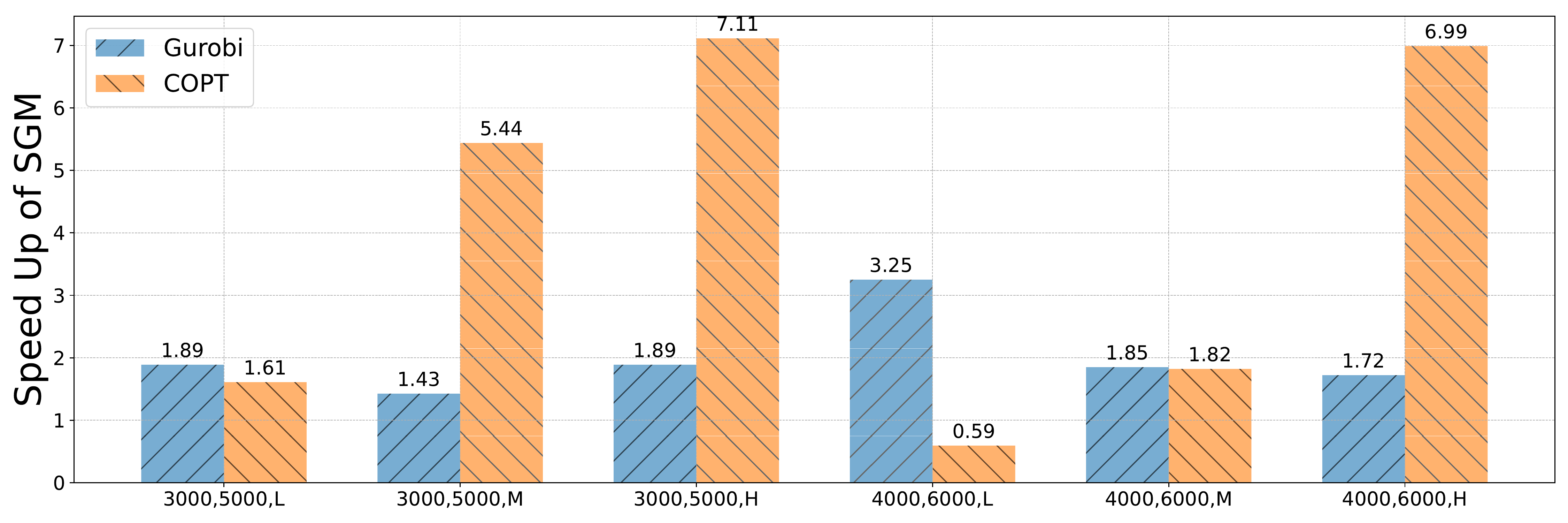}	
\caption{Experiments using logistic regression model. From top to bottom: speedup on \texttt{MKP}, \texttt{SCP} and \texttt{SCUC} instances. Left: speedup of \texttt{Gurobi}, Right: speedup of \texttt{COPT}. Each tuple in the x-axis represents $(m, n, \text{Heuristic})$ and y-axis denotes average speedup $\Sigma$ in the corresponding settings. \label{fig:mvbheur}}
\end{figure*}

\begin{figure}[h!]
\centering
\includegraphics[scale=0.25]{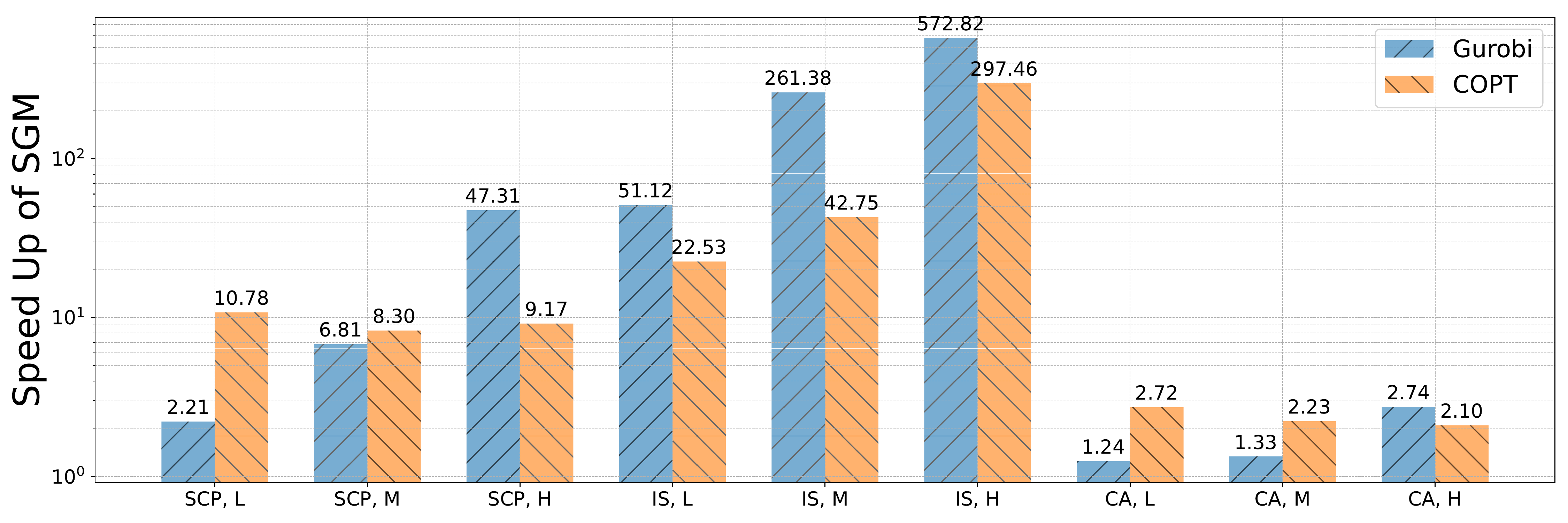}
\caption{Experiments using GNN on \texttt{SCP}, \texttt{IS}, and \texttt{CA}}
\label{fig:mvbheur-gnn}
\end{figure}

As shown in our experiments \Cref{fig:mvbheur}, when the problem structure remains fixed, even a simple logistic regression model enables {\pmvb} to deliver substantial speedups on two commercial solvers across both synthetic datasets and real-world problems. For \texttt{MKP} instances, we observe speedups of at least 500\% under several configurations. On \texttt{SCUC} problems, the method achieves an average speedup of over 50\%. Notably, for \texttt{SCP} instances, we observe consistent speedups across all test cases, with most cases exceeding 1000\%. These results highlight the practical effectiveness and robustness of our method.\\

When the problem structure is allowed to vary, \Cref{fig:mvbheur-gnn} demonstrates that the GNN model is capable of capturing structural patterns in the problem instances. However, the interaction between {\pmvb} and the solvers’ internal heuristics becomes less predictable across different datasets. For example, in the case of \texttt{SCP}, we observe a clear trend of diminishing speedup for both solvers as the problem becomes relatively easy (i.e., with small $(m,n)$). Conversely, as the problem size increases (e.g., $(m,n)=(2000,20000)$), the performance benefits of {\pmvb} become more evident.\\

Another notable observation is the solver-specific performance variation introduced by {\pmvb}. While \texttt{COPT} generally benefits more than \texttt{Gurobi} in the logistic regression setting \Cref{fig:mvbheur}, the trend reverses when GNN is used for prediction \Cref{fig:mvbheur-gnn}, where \texttt{Gurobi} sees greater improvements.

\subsection{{\pmvb} as a branching rule}

In this section, we evaluate {\GPO} as an external branching rule integrated into \texttt{Gurobi}. We implement the objective cut pruning strategy described in \Cref{sec:prune} and test the method on the following datasets:
\begin{itemize}
	\item 40 \texttt{MKP} instances using logistic regression;
\item 200 \texttt{SCP} instances and \texttt{CA} instances using GNN.
\end{itemize}
All experiments are conducted using \texttt{Gurobi} with 4 threads and a time limit of 3600 seconds. All other solver parameters are kept at their default settings. The results are summarized in \Cref{table:branching-rule}.\\

\begin{table}[h!]
\centering
\caption{Comparison of original SGM and MVB times with relative speedup. \label{table:branching-rule}}
\begin{tabular}{cccc}
\toprule
Instance & {Original time} & {\pmvb} time & Speedup \\
\midrule
\texttt{SCP} & 32.6 & 27.4 & 1.18 \\
\texttt{CA}    & 8.4  & 7.0  & 1.20 \\
\texttt{MKP-M}   & 28.7 & 25.8 & 1.11 \\
\texttt{MKP-L}   & 1283 & 979 & 1.31 \\
\bottomrule
\end{tabular}
\end{table}

According to  \Cref{table:branching-rule}, we observe a 10\% to 30\% improvement on solution time. Moreover, as we discussed in \textbf{Section \ref{sec:prune}}, the three subproblems are pruned within one second and they are almost ``free" to solve. This experiment reveals the efficiency of {\GPO} even when the underlying solver is only accessible through a blackbox. The solution time improvement is not so impressive compared to when we employ {\GPO} as a primal heuristic. However, we believe a 20\% improvement over the state-of-the-art solvers is still non-trivial, taking into account its flexibility.

\subsection{Data-free {\pmvb}}

Next, we evaluate the performance of {\GPO} as a data-free primal heuristic for MIP. In this setting, we use the interior-point solution of the LP relaxation as a surrogate for binary prediction. \texttt{COPT} is used to solve the resulting problems, with a time limit of 3600 seconds. All other solver parameters remain unchanged.

An instance is considered inapplicable for {\GPO} and excluded from the results if either of the following conditions is met:
\begin{itemize}
    \item[1)] $\mathcal{P}_{\mathcal{C}_{\mathcal{U}},
\mathcal{C}_{\mathcal{L}}}$ is immediately certified as infeasible;
\item[2)] both $\mathcal{P}_{\mathcal{C}_{\mathcal{U}},
\mathcal{C}_{\mathcal{L}}}$ and the original problem exceed the time limit.
\end{itemize}
We evaluate two types of datasets: combinatorial problem instances (\tmverbatim{SCP} and \tmverbatim{CA}) from \cite{gasse2019exact}, and benchmark MIPLIB instances from \cite{gleixner2021miplib}. The results are summarized in \Cref{fig:mvbheur-data-free}, \textbf{Table \ref{table:sgm}}, and \textbf{Table \ref{table:ins}}.

\paragraph{Combinatorial datasets.} 
\Cref{fig:mvbheur-data-free} presents the speedup results for both \tmverbatim{Gurobi} and \tmverbatim{COPT} on two types of combinatorial problems: set covering (\tmverbatim{SCP}) and combinatorial auction (\tmverbatim{CA}). As shown in the figure, the data-free version of {\pmvb} still achieves non-trivial speedups, demonstrating its effectiveness even in the absence of training data.

\begin{figure}[h]
\centering
\includegraphics[scale=0.25]{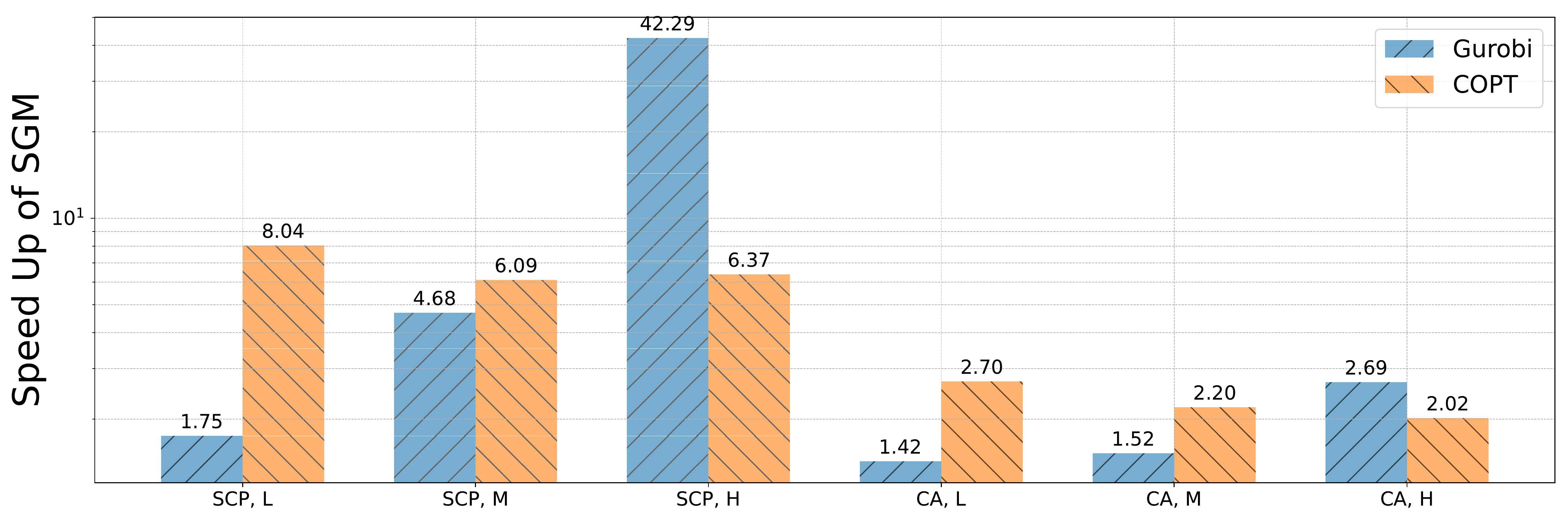}
\caption{Data-free experiments on \texttt{SCP} and \texttt{CA}}
\label{fig:mvbheur-data-free}
\end{figure}

\begin{table}[h!] 
\centering
  \caption{Shifted geometric mean for effective MIPLIB instances \label{table:sgm}}
  \begin{tabular}{cc}
    \toprule
    Parameter & SGM compared to No \tmverbatim{PreMIO}\\
    \midrule
    $\delta = 0.01, \tau = 0.9$ & $> 1.0$\\
    $\delta = 10^{- 8}, \tau = 0.9$ & 0.84\\
    $\delta = 10^{- 15}, \tau = 0.9$5 & 0.94\\
    \bottomrule
  \end{tabular}
\end{table}

\begin{table}[h!]
\centering
  \caption{Statistics of \texttt{COPT} on MIPLIB instances $\delta = 10^{- 8}, \tau =
  0.9$, $f$ denotes objective, $T$ denotes solution time\label{table:ins}}
\resizebox{0.99\columnwidth}{!}{
  \begin{tabular}{cccccccccc}
   \toprule
    Instance & $f_{\text{\tmverbatim{PreMIO}}}$ &
    $T_{\text{\tmverbatim{PreMIO}}}$ & $f_{\text{\tmverbatim{None}}}$ &
    $T_{\text{\tmverbatim{None}}}$ & Instance &
    $f_{\text{\tmverbatim{PreMIO}}}$ & $T_{\text{\tmverbatim{PreMIO}}}$ &
    $f_{\text{\tmverbatim{None}}}$ & $T_{\text{\tmverbatim{None}}}$\\
    \midrule
    academictables & 0.0e+00 & 846.9 & 1.0e+00 & 3600.0 & ci-s4 &
    3.3e+03 & 121.4 & 3.3e+03 & 101.0\\
    acc-tight5 & 0.0e+00 & 2.1 & 0.0e+00 & 4.2 & blp-ic98 & 4.5e+03 & 156.2 &
    4.5e+03 & 332.1\\
    22433 & 2.1e+04 & 0.5 & 2.1e+04 & 0.7 & brazil3 & 2.4e+02 & 70.7 & 2.4e+02
    & 87.7\\
    23588 & 8.1e+03 & 1.1 & 8.1e+03 & 0.9 & comp08-2idx & 3.7e+01 & 18.6 &
    3.7e+01 & 16.2\\
    30\_70\_45\_05\_100 & 9.0e+00 & 12.8 & 9.0e+00 & 13.6 & cmflsp50-24-8-8 &
    5.6e+07 & 2361.9 & 5.6e+07 & 2764.3\\
    amaze22012-06-28i & 0.0e+00 & 0.6 & 0.0e+00 & 0.3 &
    diameterc-msts & 7.3e+01 & 32.4 & 7.3e+01 & 30.7\\
    arki001 & 7.5e+06 & 71.9 & 7.5e+06 & 84.2 & comp07-2idx & 6.0e+00 & 321.2
    & 6.0e+00 & 468.8\\
    assign1-5-8 & 2.1e+02 & 1317.3 & 2.1e+02 & 1306.9 & aflow40b & 1.2e+03 &
    88.9 & 1.2e+03 & 121.0\\
    a1c1s1 & 1.2e+04 & 204.1 & 1.2e+04 & 217.7 & cap6000 & -2.5e+06 & 0.6 &
    -2.5e+06 & 0.3\\
    blp-ar98 & 6.2e+03 & 293.6 & 6.2e+03 & 340.7 & ab71-20-100 & -1.0e+10 &
    3.3 & -1.0e+10 & 3.5\\
    ab51-40-100 & -1.0e+10 & 6.1 & -1.0e+10 & 6.5 & beavma & 3.8e+05 & 0.2 &
    3.8e+05 & 0.2\\
    a2c1s1 & 1.1e+04 & 333.2 & 1.1e+04 & 352.5 & cost266-UUE & 2.5e+07 &
    1549.1 & 2.5e+07 & 1574.1\\
    blp-ar98 & 6.2e+03 & 293.6 & 6.2e+03 & 340.7 & csched007 & 3.5e+02 & 383.2
    & 3.5e+02 & 408.5\\
    blp-ic97 & 4.0e+03 & 441.4 & 4.0e+03 & 465.4 & csched008 & 1.7e+02 & 80.2
    & 1.7e+02 & 255.3\\
    aflow30a & 1.2e+03 & 4.4 & 1.2e+03 & 4.5 & danoint & 6.6e+01 & 650.5 &
    6.6e+01 & 484.3\\
    \bottomrule
  \end{tabular}}
\end{table}

\paragraph{MIPLIB instances.}
We ultimately obtain 57 MIP instances for evaluation. As shown in \textbf{Table \ref{table:sgm}} and \textbf{Table \ref{table:ins}}, we observe an improvement of up to 16\% in SGM across these instances. Several interesting observations emerge:
\begin{itemize}[leftmargin=*]
\item Some instances (e.g., \texttt{academictable}) become solvable on the primal side when solving $\mathcal{P}_{\mathcal{C}_{\mathcal{U}},
\mathcal{C}_{\mathcal{L}}}$.  Moreover, when this problem is feasible, it rarely compromises the optimality of the original problem.
\item It's possible for $\mathcal{P}_{\mathcal{C}_{\mathcal{U}},
\mathcal{C}_{\mathcal{L}}}$ to be infeasible when using the LP solution, especially if the LP relaxation is weak. In such cases, the relaxation may mislead the direction of the solution.
\item Very conservative {\GPO} parameter settings are required. As suggested by \textbf{Table \ref{table:sgm}}, achieving competitive performance in the data-free case requires setting the error parameter 
$\delta$ very close to 0. This contrasts with the data-driven case, where values of $\delta$ up to 0.3 can be used without triggering infeasibility.
\end{itemize}
As a data-independent heuristic, our method also shows promising potential for general MIP solving. We plan to actively explore this direction in future work.

%% file: conclusion.tex
\section{Conclusions}\label{sec:conclusion}
In this paper, we propose {\GPO}, a simple yet effective method for accelerating MIP solving using machine learning models. Grounded in learning theory, our approach draws a connection between concentration inequalities and multi-variable branching hyperplanes, which informs the design of a data-driven branching procedure. {\GPO} offers an interpretable mechanism for incorporating prior knowledge from offline solves into online MIP solving. Furthermore, our formulation—which connects binary random variables with binary predictions—may be of independent interest for the development of future data-driven MIP heuristics.

%% file: appendix.tex
\section{Auxiliary lemmas}\label{appendix:lemma}

\begin{lem}[Hoeffding inequality] \label{lem:hoeffding}
Given independent random variables $Y_1, \ldots, Y_n$ such that $Y_i \in [0,
  1]$, then for any $t > 0$, we have
  \[ \mathbb{P} \{ \textstyle \sum_{i = 1}^n (Y_i -\mathbb{E} [Y_i]) \geq t \}
     \leq  e^{- 2 t^2 / n} . \]
\end{lem} 

\begin{lem}[Bernstein inequality] \label{lem:bernstein}
Given i.i.d. Bernoulli random variables $Y_1, \ldots, Y_n$ with $Y_i \in \{0, 1\}$ and $\mathbb{E}[Y_i] = p$, then for any $t > 0$, we have
  \[ \mathbb{P} \{ \textstyle \sum_{i = 1}^n Y_i - np \geq t \}
     \leq  \exp(-\frac{t^2}{2(np + t/3)}) . \]
\end{lem} 

\begin{lem}[Chebyshev inequality] \label{lem:chebyshev}
Given a random variance $Y$ such that $\mathbb{V}[Y] < \infty$, then for any $t > 0$, we have 
\[\Pbb\{|Y - \Ebb[Y]| \geq t\} \leq \tfrac{\mathbb{V}[Y]}{t^2}. \]
\end{lem}

\begin{lem}\label{lem:concentrate}
    Given a sequence $\{\lambda_i, i = 1,2,\dots, n\}$ which are independent and uniformally distributed in $(0,1)$, for any fixed $\delta \in (0,1)$ and any $1 \leq j \leq 1 /\delta$, define the set $I_j = \{ i: \lambda_i \in [(j-1)\delta, j\delta]$, we have that with probability at least $1 -\lceil \frac{1}{\delta} \rceil \exp(\frac{-n\delta}{4})$, we have
    $$
    |I_j | \leq 2n\delta, \quad \forall 1 \leq j \leq \lceil 1 /\delta \rceil.
    $$
    Besides, for all subset $S \subseteq [n]$ satisfies $|S| \geq 4n\delta$, we have $\textstyle \sum_{i \in S} \lambda_i \geq \frac{|S|^2}{8n}$.
\end{lem}

\begin{proof}
	For any fixed $j$, we know that $\mathbb{P}(\lambda_i \in [(j-1)\delta, j\delta]) = \delta$ for all $i \in [n]$. Then, by applying \Cref{lem:bernstein} with $p = \delta$ and $t = n\delta$, we could show that \begin{equation}\label{ineq:concentration:uniform}
        \Pbb\{|\{i: \lambda_i \in [(j-1)\delta, j\delta]  \}| \geq 2 n \delta \}  \leq \exp(\tfrac{-n\delta}{4}).
    \end{equation}
    According to the union bound, \eqref{ineq:concentration:uniform} holds for all $1\leq i \leq 1/\delta$ with probability at least $1 - \frac{1}{\delta} \exp(\frac{-n\delta}{4})$. Besides, consider the minimal of the sum $\textstyle \sum_{i \in S} \lambda_i$ among all subset $S \subseteq [n]$, which is achieved as we choose $S = I_1 \cup I_2 \cup \cdots \cup I_k$ for some $k$ and each $|I_j|$ is as large as possible, i.e., $|I_j| = 2n\delta$ and $k = |S|/(2n\delta)$. So we have
    $$
    \textstyle \sum_{i \in S} \lambda_i \geq  \textstyle \sum_{j =1}^{|S|/(2n\delta)} 2(j-1)n\delta^2 \geq \frac{|S|}{4n}(|S| - 2n\delta)  \geq \frac{|S|^2}{8n}
    $$
    where the last inequality holds because we assume $|S| \geq 4n\delta$.
\end{proof}

\section{Proof of Theorem \ref{thm:mvb}}
The side $j \in \mathcal{L}$ follows by
  \begin{equation*}
      \begin{aligned}
    \mathbb{P} \{ \textstyle \sum_{j \in \mathcal{L}} y_j^{\star} (\xi) \leq \eta
    \} ={} & \mathbb{P} \{ \textstyle \sum_{j \in \mathcal{L}} y_j^{\star} (\xi)
    - \hat{y}_j (\xi) \leq \eta \} \nonumber\\
    ={} & \mathbb{P} \{ \textstyle \sum_{j \in \mathcal{L}} 1 -\mathbb{I} \{
    y^{\star}_j (\xi) = \hat{y}_j (\xi) \} \leq \eta \} \nonumber\\
    ={} & \mathbb{P} \{ \textstyle \sum_{j \in \mathcal{L}} \mathbb{I} \{ y^{\star}_j
    (\xi) = \hat{y}_j (\xi) \} \geq | \mathcal{L} | - \eta \}
    \nonumber\\
    ={} & \mathbb{P} \{ \textstyle \sum_{j \in \mathcal{L}} \mathbb{I} \{ y^{\star}_j
    (\xi) = \hat{y}_j (\xi) \} - (1 - \delta) \textstyle \sum_{j \in \mathcal{L}}
    \Delta_j^{\delta} \geq | \mathcal{L} | - \eta - (1 - \delta) \textstyle \sum_{j \in
    \mathcal{L}} \Delta_j^{\delta} \} \nonumber\\
    \geq{} & \mathbb{P} \{ \textstyle \sum_{j \in \mathcal{L}} \mathbb{I} \{
    y^{\star}_j (\xi) = \hat{y}_j (\xi) \} - \pi_j \geq | \mathcal{L} | - \eta
    - (1 - \delta) \textstyle \sum_{j \in \mathcal{L}} \Delta_j^{\delta} \}
    \nonumber\\
    \geq{} & 1 -\mathbb{P} \{ \textstyle \sum_{j \in \mathcal{L}} \mathbb{I} \{
    y^{\star}_j (\xi) = \hat{y}_j (\xi) \} - \pi_j \leq | \mathcal{L} | - \eta
    - (1 - \delta) \textstyle \sum_{j \in \mathcal{L}} \Delta_j^{\delta} \}
    \nonumber\\
    \geq{} & 1 - \exp ( - 2\tfrac{ [| \mathcal{L} | - \eta - (1 - \delta)
     \sum_{j \in \mathcal{L}} \Delta_j^{\delta} ]}{| \mathcal{L} |} ) .
    \nonumber
\end{aligned}
  \end{equation*}
  Taking $\eta = | \mathcal{L} | - (1 - \delta) \textstyle \sum_{j \in \mathcal{L}}
  \Delta_j^{\delta} + \gamma$ completes the proof.

%% file: premio_arxiv.bbl
\begin{thebibliography}{10}

\bibitem{alvarez2017machine}
Alejandro~Marcos Alvarez, Quentin Louveaux, and Louis Wehenkel.
\newblock A machine learning-based approximation of strong branching.
\newblock {\em INFORMS Journal on Computing}, 29:185--195, 2017.

\bibitem{balcan2018learning}
Maria-Florina Balcan, Travis Dick, Tuomas Sandholm, and Ellen Vitercik.
\newblock Learning to branch.
\newblock In {\em International conference on machine learning}, pages
  344--353. PMLR, 2018.

\bibitem{balcan2022improved}
Maria-Florina Balcan, Siddharth Prasad, Tuomas Sandholm, and Ellen Vitercik.
\newblock Improved sample complexity bounds for branch-and-cut.
\newblock In {\em 28th International Conference on Principles and Practice of
  Constraint Programming (CP 2022)}, 2022.

\bibitem{balcan2022structural}
Maria-Florina~F Balcan, Siddharth Prasad, Tuomas Sandholm, and Ellen Vitercik.
\newblock Structural analysis of branch-and-cut and the learnability of gomory
  mixed integer cuts.
\newblock {\em Advances in Neural Information Processing Systems},
  35:33890--33903, 2022.

\bibitem{benati2007mixed}
Stefano Benati and Romeo Rizzi.
\newblock A mixed integer linear programming formulation of the optimal
  mean/value-at-risk portfolio problem.
\newblock {\em European Journal of Operational Research}, 176(1):423--434,
  2007.

\bibitem{bengio2004no}
Yoshua Bengio and Yves Grandvalet.
\newblock No unbiased estimator of the variance of k-fold cross-validation.
\newblock {\em Journal of machine learning research}, 5(Sep):1089--1105, 2004.

\bibitem{berthold2021learning}
Timo Berthold and Gregor Hendel.
\newblock Learning to scale mixed-integer programs.
\newblock In {\em Proceedings of the AAAI Conference on Artificial
  Intelligence}, volume~35, pages 3661--3668, 2021.

\bibitem{berthold2018four}
Timo Berthold, Michael Perregaard, and Csaba M{\'e}sz{\'a}ros.
\newblock Four good reasons to use an interior point solver within a mip
  solver.
\newblock In {\em Operations Research Proceedings 2017: Selected Papers of the
  Annual International Conference of the German Operations Research Society
  (GOR), Freie Universi{\"a}t Berlin, Germany, September 6-8, 2017}, pages
  159--164. Springer, 2018.

\bibitem{birchfield2016grid}
Adam~B Birchfield, Ti~Xu, Kathleen~M Gegner, Komal~S Shetye, and Thomas~J
  Overbye.
\newblock Grid structural characteristics as validation criteria for synthetic
  networks.
\newblock {\em IEEE Transactions on power systems}, 32(4):3258--3265, 2016.

\bibitem{canturk2024scalable}
Furkan Cant{\"u}rk, Taha Varol, Reyhan Aydo{\u{g}}an, and Okan~{\"O}rsan
  {\"O}zener.
\newblock Scalable primal heuristics using graph neural networks for
  combinatorial optimization.
\newblock {\em Journal of Artificial Intelligence Research}, 80:327--376, 2024.

\bibitem{chu1998genetic}
Paul~C Chu and John~E Beasley.
\newblock A genetic algorithm for the multidimensional knapsack problem.
\newblock {\em Journal of heuristics}, 4:63--86, 1998.

\bibitem{cplex2009v12}
IBM~ILOG Cplex.
\newblock V12. 1: User’s manual for cplex.
\newblock {\em International Business Machines Corporation}, 46(53):157, 2009.

\bibitem{ding2020accelerating}
Jian-Ya Ding, Chao Zhang, Lei Shen, Shengyin Li, Bing Wang, Yinghui Xu, and
  Le~Song.
\newblock Accelerating primal solution findings for mixed integer programs
  based on solution prediction.
\newblock In {\em Proceedings of the aaai conference on artificial
  intelligence}, volume~34, pages 1452--1459, 2020.

\bibitem{fischetti2003local}
Matteo Fischetti and Andrea Lodi.
\newblock Local branching.
\newblock {\em Mathematical programming}, 98:23--47, 2003.

\bibitem{freund1997decision}
Yoav Freund and Robert~E Schapire.
\newblock A decision-theoretic generalization of on-line learning and an
  application to boosting.
\newblock {\em Journal of computer and system sciences}, 55(1):119--139, 1997.

\bibitem{gamrath2015branching}
Gerald Gamrath, Anna Melchiori, Timo Berthold, Ambros~M Gleixner, and Domenico
  Salvagnin.
\newblock Branching on multi-aggregated variables.
\newblock In {\em Integration of AI and OR Techniques in Constraint
  Programming: 12th International Conference, CPAIOR 2015, Barcelona, Spain,
  May 18-22, 2015, Proceedings 12}, pages 141--156. Springer, 2015.

\bibitem{gasse2019exact}
Maxime Gasse, Didier Ch{\'e}telat, Nicola Ferroni, Laurent Charlin, and Andrea
  Lodi.
\newblock Exact combinatorial optimization with graph convolutional neural
  networks.
\newblock {\em Advances in neural information processing systems}, 32, 2019.

\bibitem{ge2022cardinal}
Dongdong Ge, Qi~Huangfu, Zizhuo Wang, Jian Wu, and Yinyu Ye.
\newblock Cardinal optimizer (copt) user guide.
\newblock {\em arXiv preprint arXiv:2208.14314}, 2022.

\bibitem{gilmour1996interpretation}
Steven~G Gilmour.
\newblock The interpretation of mallows’s cp-statistic.
\newblock {\em Journal of the Royal Statistical Society Series D: The
  Statistician}, 45(1):49--56, 1996.

\bibitem{gleixner2021miplib}
Ambros Gleixner, Gregor Hendel, Gerald Gamrath, Tobias Achterberg, Michael
  Bastubbe, Timo Berthold, Philipp Christophel, Kati Jarck, Thorsten Koch, Jeff
  Linderoth, et~al.
\newblock Miplib 2017: data-driven compilation of the 6th mixed-integer
  programming library.
\newblock {\em Mathematical Programming Computation}, 13(3):443--490, 2021.

\bibitem{gupta2020hybrid}
Prateek Gupta, Maxime Gasse, Elias Khalil, Pawan Mudigonda, Andrea Lodi, and
  Yoshua Bengio.
\newblock Hybrid models for learning to branch.
\newblock {\em Advances in neural information processing systems},
  33:18087--18097, 2020.

\bibitem{gurobi2021gurobi}
LLC Gurobi~Optimization.
\newblock Gurobi optimizer reference manual, 2021.

\bibitem{han2023gnn}
Qingyu Han, Linxin Yang, Qian Chen, Xiang Zhou, Dong Zhang, Akang Wang, Ruoyu
  Sun, and Xiaodong Luo.
\newblock A gnn-guided predict-and-search framework for mixed-integer linear
  programming.
\newblock {\em arXiv preprint arXiv:2302.05636}, 2023.

\bibitem{jenkins1982parametric}
Larry Jenkins.
\newblock Parametric mixed integer programming: an application to solid waste
  management.
\newblock {\em Management Science}, 28(11):1270--1284, 1982.

\bibitem{karamanov2011branching}
Miroslav Karamanov and G{\'e}rard Cornu{\'e}jols.
\newblock Branching on general disjunctions.
\newblock {\em Mathematical Programming}, 128(1-2):403--436, 2011.

\bibitem{khalil2016learning}
Elias Khalil, Pierre Le~Bodic, Le~Song, George Nemhauser, and Bistra Dilkina.
\newblock Learning to branch in mixed integer programming.
\newblock In {\em Proceedings of the AAAI Conference on Artificial
  Intelligence}, volume~30, 2016.

\bibitem{khalil2022mip}
Elias~B Khalil, Christopher Morris, and Andrea Lodi.
\newblock Mip-gnn: A data-driven framework for guiding combinatorial solvers.
\newblock In {\em Proceedings of the AAAI Conference on Artificial
  Intelligence}, volume~36, pages 10219--10227, 2022.

\bibitem{kohli2004average}
Rajeev Kohli, Ramesh Krishnamurti, and Prakash Mirchandani.
\newblock Average performance of greedy heuristics for the integer knapsack
  problem.
\newblock {\em European Journal of Operational Research}, 154(1):36--45, 2004.

\bibitem{lee2019learning}
Mengyuan Lee, Guanding Yu, and Geoffrey~Ye Li.
\newblock Learning to branch: Accelerating resource allocation in wireless
  networks.
\newblock {\em IEEE Transactions on Vehicular Technology}, 69(1):958--970,
  2019.

\bibitem{lin2022learning}
Jiacheng Lin, Jialin Zhu, Huangang Wang, and Tao Zhang.
\newblock Learning to branch with tree-aware branching transformers.
\newblock {\em Knowledge-Based Systems}, 252:109455, 2022.

\bibitem{liu2022learning}
Defeng Liu, Matteo Fischetti, and Andrea Lodi.
\newblock Learning to search in local branching.
\newblock In {\em Proceedings of the aaai conference on artificial
  intelligence}, volume~36, pages 3796--3803, 2022.

\bibitem{lueker1982average}
George~S Lueker.
\newblock On the average difference between the solutions to linear and integer
  knapsack problems.
\newblock In {\em Applied Probability-Computer Science: The Interface Volume
  1}, pages 489--504. Springer.

\bibitem{mittelmann2020benchmarking}
Hans~D Mittelmann.
\newblock Benchmarking optimization software-a (hi) story.
\newblock In {\em SN operations research forum}, volume~1, page~2. Springer,
  2020.

\bibitem{morales2020expected}
Fernando~A Morales and Jairo~A Mart{\'\i}nez.
\newblock Expected performance and worst case scenario analysis of the
  divide-and-conquer method for the 0-1 knapsack problem.
\newblock {\em arXiv preprint arXiv:2008.04124}, 2020.

\bibitem{nair2020solving}
Vinod Nair, Sergey Bartunov, Felix Gimeno, Ingrid Von~Glehn, Pawel Lichocki,
  Ivan Lobov, Brendan O'Donoghue, Nicolas Sonnerat, Christian Tjandraatmadja,
  Pengming Wang, et~al.
\newblock Solving mixed integer programs using neural networks.
\newblock {\em arXiv preprint arXiv:2012.13349}, 2020.

\bibitem{nazari2018reinforcement}
Mohammadreza Nazari, Afshin Oroojlooy, Lawrence Snyder, and Martin Tak{\'a}c.
\newblock Reinforcement learning for solving the vehicle routing problem.
\newblock {\em Advances in neural information processing systems}, 31, 2018.

\bibitem{optimizer2021v8}
FICO~Xpress Optimizer.
\newblock v8. 11 reference manual, 2021.

\bibitem{pochet2006production}
Yves Pochet and Laurence~A Wolsey.
\newblock {\em Production planning by mixed integer programming}, volume 149.
\newblock Springer, 2006.

\bibitem{qu2022improved}
Qingyu Qu, Xijun Li, Yunfan Zhou, Jia Zeng, Mingxuan Yuan, Jie Wang, Jinhu Lv,
  Kexin Liu, and Kun Mao.
\newblock An improved reinforcement learning algorithm for learning to branch.
\newblock {\em arXiv preprint arXiv:2201.06213}, 2022.

\bibitem{reid2009surrogate}
Mark~D Reid and Robert~C Williamson.
\newblock Surrogate regret bounds for proper losses.
\newblock In {\em Proceedings of the 26th Annual International Conference on
  Machine Learning}, pages 897--904, 2009.

\bibitem{sahinidis2019mixed}
Nikolaos~V Sahinidis.
\newblock Mixed-integer nonlinear programming 2018, 2019.

\bibitem{scavuzzo2022learning}
Lara Scavuzzo, Feng Chen, Didier Ch{\'e}telat, Maxime Gasse, Andrea Lodi, Neil
  Yorke-Smith, and Karen Aardal.
\newblock Learning to branch with tree mdps.
\newblock {\em Advances in Neural Information Processing Systems},
  35:18514--18526, 2022.

\bibitem{song2020general}
Jialin Song, Yisong Yue, Bistra Dilkina, et~al.
\newblock A general large neighborhood search framework for solving integer
  linear programs.
\newblock {\em Advances in Neural Information Processing Systems},
  33:20012--20023, 2020.

\bibitem{sonnerat2021learning}
Nicolas Sonnerat, Pengming Wang, Ira Ktena, Sergey Bartunov, and Vinod Nair.
\newblock Learning a large neighborhood search algorithm for mixed integer
  programs.
\newblock {\em arXiv preprint arXiv:2107.10201}, 2021.

\bibitem{umetani2017exploiting}
Shunji Umetani.
\newblock Exploiting variable associations to configure efficient local search
  algorithms in large-scale binary integer programs.
\newblock {\em European Journal of Operational Research}, 263(1):72--81, 2017.

\bibitem{vapnik2013nature}
Vladimir Vapnik.
\newblock {\em The nature of statistical learning theory}.
\newblock Springer science \& business media, 2013.

\bibitem{vershynin2018high}
Roman Vershynin.
\newblock {\em High-dimensional probability: An introduction with applications
  in data science}, volume~47.
\newblock Cambridge university press, 2018.

\bibitem{wolsey2020integer}
Laurence~A Wolsey.
\newblock {\em Integer programming}.
\newblock John Wiley \& Sons, 2020.

\bibitem{wolsey1999integer}
Laurence~A Wolsey and George~L Nemhauser.
\newblock {\em Integer and combinatorial optimization}, volume~55.
\newblock John Wiley \& Sons, 1999.

\bibitem{wood2013power}
Allen~J Wood, Bruce~F Wollenberg, and Gerald~B Shebl{\'e}.
\newblock {\em Power generation, operation, and control}.
\newblock John Wiley \& Sons, 2013.

\bibitem{wu2021learning}
Yaoxin Wu, Wen Song, Zhiguang Cao, and Jie Zhang.
\newblock Learning large neighborhood search policy for integer programming.
\newblock {\em Advances in Neural Information Processing Systems},
  34:30075--30087, 2021.

\bibitem{xavier2021learning}
{\'A}linson~S Xavier, Feng Qiu, and Shabbir Ahmed.
\newblock Learning to solve large-scale security-constrained unit commitment
  problems.
\newblock {\em INFORMS Journal on Computing}, 33(2):739--756, 2021.

\bibitem{xu2017application}
Ti~Xu, Adam~B Birchfield, Kathleen~M Gegner, Komal~S Shetye, and Thomas~J
  Overbye.
\newblock Application of large-scale synthetic power system models for energy
  economic studies.
\newblock 2017.

\bibitem{yang2021multivariable}
Yu~Yang, Natashia Boland, and Martin Savelsbergh.
\newblock Multivariable branching: A 0-1 knapsack problem case study.
\newblock {\em INFORMS Journal on Computing}, 33(4):1354--1367, 2021.

\bibitem{zarpellon2021parameterizing}
Giulia Zarpellon, Jason Jo, Andrea Lodi, and Yoshua Bengio.
\newblock Parameterizing branch-and-bound search trees to learn branching
  policies.
\newblock In {\em Proceedings of the aaai conference on artificial
  intelligence}, volume~35, pages 3931--3939, 2021.

\end{thebibliography}
